\documentclass[10pt, twoside]{amsart}

\usepackage{t1enc}
\usepackage[latin1]{inputenc}
\usepackage{latexsym}
\usepackage{amssymb}
\usepackage{graphicx}
\usepackage{amsmath}
\usepackage{amsthm}
\usepackage{amsfonts}
\usepackage{mathpazo}
\usepackage{mathrsfs}
\usepackage[paper = letterpaper, left = 3.0cm, right = 3.0cm, headsep = 6mm, footskip = 10mm,
top = 28mm, bottom = 30mm, footnotesep=5mm, headheight = 2cm]{geometry}
\usepackage{fancyhdr}
\usepackage[all]{xy}
\usepackage[british]{babel}
\usepackage{soul}
\usepackage{hyperref}
\newcommand*{\embeds}{\hookrightarrow}
\newcommand*{\widebar}{\overline}
\newcommand*{\definiere}{\mathrel{\mathop:}=}
\newcommand*{\rdefiniere}{=\mathrel{\mathop:}}
\newcommand*{\tensor}{\otimes}

\newcommand{\cyclic}{\mathop{\kern0.9ex{{+}\kern-2.10ex\raise-0.20
      ex\hbox{\Large\hbox{$\circlearrowright$}}}}\limits}
\newcommand{\acts}{\mbox{ \raisebox{0.26ex}{\tiny{$\bullet$}} }}

\def\N{\ifmmode{\mathbb N}\else{$\mathbb N$}\fi}
\def\Z{\ifmmode{\mathbb Z}\else{$\mathbb Z$}\fi}
\def\Q{\ifmmode{\mathbb Q}\else{$\mathbb Q$}\fi}
\def\R{\ifmmode{\mathbb R}\else{$\mathbb R$}\fi}
\def\C{\ifmmode{\mathbb C}\else{$\mathbb C$}\fi}
\def\K{\ifmmode{\mathbb K}\else{$\mathbb K$}\fi}
\def\P{\ifmmode{\mathbb P}\else{$\mathbb P$}\fi}
\def\g{\ifmmode{\mathfrak g}\else {$\mathfrak g$}\fi}
\def\h{\ifmmode{\mathfrak h}\else {$\mathfrak h$}\fi}
\def\a{\ifmmode{\mathfrak a}\else {$\mathfrak a$}\fi}
\def\k{\ifmmode{\mathfrak k}\else {$\mathfrak k$}\fi}
\def\p{\ifmmode{\mathfrak p}\else {$\mathfrak p$}\fi}
\def\b{\ifmmode{\mathfrak b}\else {$\mathfrak b$}\fi}
\def\n{\ifmmode{\mathfrak n}\else {$\mathfrak n$}\fi}
\def\m{\ifmmode{\mathfrak m}\else {$\mathfrak m$}\fi}
\def\t{\ifmmode{\mathfrak t}\else {$\mathfrak t$}\fi}
\def\O{\ifmmode{\mathscr{O}}\else {$\mathscr{O}$}\fi}
\def\W{\ifmmode{\mathcal{V}}\else {$\mathscr{W}$}\fi}
\def\id{{\rm id}}

\def\hq{/\hspace{-0.14cm}/}
\def\kleinematrix#1,#2,#3,#4,{\begin{pmatrix}#1 & #2 \\ #3 & #4
  \end{pmatrix}}

\DeclareMathOperator{\Ad}{Ad}
\DeclareMathOperator{\Id}{Id}
\DeclareMathOperator{\Lie}{Lie}
\DeclareMathOperator{\trdeg}{trdeg}
\DeclareMathOperator{\pr}{pr}
\DeclareMathOperator{\End}{End}
\DeclareMathOperator{\dom}{dom}

\DeclareMathOperator{\supp}{supp}
\DeclareMathOperator{\Type}{Type}
\DeclareMathOperator{\Proj}{Proj}
\DeclareMathOperator{\Spec}{Spec}

\newtheoremstyle{daniel}{3.0mm}{0mm}{\itshape}{}{\bfseries}{.}{2mm}{}
\theoremstyle{daniel}
\newtheorem{thm}{Theorem}[section]
\newtheorem{prop}[thm]{Proposition}
\newtheorem{Defi}[thm]{Definition}
\newtheorem{lemma}[thm]{Lemma}

\newtheorem{Exs}[thm]{Examples}
\newtheorem{Ex}[thm]{Example}
\newtheorem{Rems}[thm]{Remarks}
\newtheorem{Rem}[thm]{Remark}
\newtheorem{thn}{Theorem}
\newtheorem*{thm*}{Theorem}

\newtheorem*{prop*}{Proposition}

\newenvironment{rem}   {\begin{Rem}\em}{\end{Rem}}
\newenvironment{rems}   {\begin{Rems}\em}{\end{Rems}}

\newenvironment{ex}  {\begin{Ex}\em}{\end{Ex}}

\def\cprime{$'$} \def\polhk#1{\setbox0=\hbox{#1}{\ooalign{\hidewidth
  \lower1.5ex\hbox{`}\hidewidth\crcr\unhbox0}}}
  \def\polhk#1{\setbox0=\hbox{#1}{\ooalign{\hidewidth
  \lower1.5ex\hbox{`}\hidewidth\crcr\unhbox0}}}
\providecommand{\bysame}{\leavevmode\hbox to3em{\hrulefill}\thinspace}
\providecommand{\MR}{\relax\ifhmode\unskip\space\fi MR }
\providecommand{\MRhref}[2]{%
  \href{http://www.ams.org/mathscinet-getitem?mr=#1}{#2}
}
\providecommand{\href}[2]{#2}
\begin{document}
\title{Projectivity of analytic Hilbert and Kaehler quotients}
\author{Daniel Greb}
\date{}
\begin{abstract}
We investigate algebraicity properties of quotients of complex spaces by complex re\-duc\-tive Lie groups $G$. We obtain a projectivity result for compact momentum map quotients of algebraic $G$-varieties. Furthermore, we prove equivariant versions of Kodaira's Embedding Theorem and Chow's Theorem relative to an analytic Hilbert quotient. Combining these results we derive an equivariant algebraisation theorem for complex spaces with projective quotient.
\end{abstract}
\maketitle

\section*{Introduction}
Symplectic reduction and holomorphic invariant theory have been very successfully applied to construct and study quotient spaces for actions of reductive groups on complex spaces. Given a complex reductive Lie group $G$, a holomorphic $G$-space $X$, and an equivariant momentum map $\mu: X \to \Lie(K)^*$ for the action of a maximal compact subgroup $K$ of $G$ with respect to a $K$-invariant K\"ahler structure on $X$,  the fundamental link between symplectic and complex geometry is given by the concept of semistability: generalising fundamental work of Kirwan \cite{KirwanCohomology}, it was shown by Heinzner and Loose \cite{ReductionOfHamiltonianSpaces} that the set of \emph{$\mu$-semistable
points}
\[X(\mu) \definiere \{ x \in X \mid \overline{G\acts x} \cap \mu^{-1}(0) \neq \emptyset\}\] admits an \emph{analytic Hilbert quotient}, i.e., a $G$-invariant holomorphic Stein map $\pi: X(\mu)\to
X(\mu)\hq G$ onto a complex space $X(\mu)\hq G$ with structure sheaf $\mathscr{O}^{hol}_{X(\mu)/\negthickspace / G}=(\pi_*\mathscr{O}_{X(\mu)}^{hol})^G$, see also \cite{GuilleminGeometricQuantisation} and \cite{SjamaarSlices}. If the complex space $X$ is projective algebraic and the
K\"ahler structure under consideration is the Fubini-Study form, the quotient $X(\mu)\hq G$ is projective algebraic and can also be constructed via Mumford's Geometric Invariant Theory, cf.\ \cite{MumfordGIT}.

Motivated by this observation, in the first part of this note we study \emph{algebraic Hamiltonian $G$-varieties}, complex algebraic $G$-varieties with $K$-invariant K\"ahler structures and $K$-equivariant momentum maps $\mu: X \to \Lie(K)^*$. We investigate the question whether the momentum map quotient $X(\mu)\hq G$ of an algebraic Hamiltonian $G$-variety $X$ inherits algebraicity properties from $X$. As the main result of this part, we prove
\begin{thn}[Projectivity Theorem]\label{1}
Let $G$ be a complex reductive Lie group and $K$ a maximal compact subgroup of $G$. Let $X$ be a $G$-irreducible algebraic Hamiltonian $G$-variety with momentum map $\mu: X
\rightarrow \Lie(K)^*$. Assume that $X$ has only $\mathit{1}$-rational singularities. Then, if $\mu^{-1}(0)$ is compact, the analytic Hilbert quotient $X(\mu)\hq G$ is a projective algebraic variety.
\end{thn}
\vspace{1mm}
This generalises earlier results by Sjamaar \cite{SjamaarSlices} who proved projectivity of
$X(\mu)\hq G$ under the additional assumptions that $X$ is smooth and that the K\"ahler structure is the curvature form of a
positive $G$-linearised line bundle, and by Heinzner and
Migliorini \cite{MomentumProjectivity} who dealt with the case of arbitrary K\"ahler forms on
projective manifolds. We emphasise at this point that our result neither assumes the K\"ahler form to be integral nor requires the $G$-variety to be smooth or complete.

%We emphasise at this point that our result does not assume the K\"ahler form
%to be integral, and does also hold for actions on non-complete and singular algebraic varieties.

In contrast to the GIT-situation, there exist examples of Hamiltonian $G$-spaces such that the set of $\mu$-semistable points is not analytically Zariski-open, see Section~\ref{section:Chow}. Using Theorem \ref{1} and the Chow quotient of an algebraic $G$-variety, we show that this phenomenon cannot occur in our setup:
\begin{thn}\label{unstablepointsarealgebraic}
Let $G$ be a complex reductive Lie group with maximal compact subgroup $K$. Let $X$ be a
$G$-irreducible algebraic Hamiltonian $G$-variety with momentum map
$\mu: X \to \Lie (K)^*$. Assume that $X$ has $\mathit{1}$-rational singularities and that $\mu^{-1}(0)$ is compact.
Then, the set $X(\mu)$ of $\mu$-semistable points is algebraically Zariski-open in $X$.
\end{thn}
\vspace{1mm}
Complex reductive Lie groups and their representations are classically known to have strong
algebraicity properties. As a consequence, actions of
%complex reductive
these groups on complex spaces can often be studied by methods that are closely related to classical invariant theory. For example, it was shown by Snow \cite{Snow} that each fibre of an analytic Hilbert quotient carries a natural affine algebraic structure given by the algebra of $K$-finite holomorphic functions. With respect to this algebraic structure the action of $G$ is also algebraic. In the second part of this paper
we use representation theory relative to an analytic Hilbert quotient, separation properties of $K$-finite holomorphic functions and classical
vanishing theorems to prove the following generalisation of Snow's result.
\begin{thn}[Algebraicity Theorem]\label{2}
Let $G$ be a complex reductive Lie group. Let $X$ be a holomorphic $G$-space such that the analytic
Hilbert quotient $\pi: X \rightarrow X\hq
G$ exists and such that $X\hq G$ is projective algebraic. Then, up to $G$-equivariant algebraic isomorphisms, there exists a uniquely determined
quasi-projective algebraic $G$-variety $Z$ with algebraic Hilbert quotient $\pi_Z: Z \rightarrow
Z\hq G$ such that $X$ is $G$-equivariantly biholomorphic to $Z$.
\end{thn}
Here, an \emph{algebraic Hilbert quotient} of a complex algebraic $G$-variety $Z$ is a $G$-invariant, affine morphism $\pi: Z \to Z\hq G$ onto an algebraic variety $Z\hq G$ with $\mathscr{O}_{Z/\negthickspace/G} = (\pi_* \mathscr{O}_Z)^G$. We derive Theo\-rem~\ref{2} from analogues of Kodaira's Embedding Theorem and Chow's Theorem in our setup (Theo\-rem~\ref{embedding} and Theorem~\ref{algebraicitytheorem}).

In summary, the results proven in this note establish a close connection between algebraicity properties of complex $G$-spaces and those of associated analytic Hilbert quotients. Furthermore, they underline the relevance of classical invariant theoretic methods in the study of group actions of complex reductive groups on complex spaces.

%Together with the results obtained in the first part of the paper it underlines the close connection between algebraicity properties of $G$-spaces and algebraicity properties of the associated analytic Hilbert quotients.

\subsection*{Outline of the paper}
In Sections 1 and 2 we recall definitions related to and basic properties of actions of complex reductive Lie groups. The proof of Theorem~\ref{1} uses a result of Namikawa \cite{ProjectivityofMoishezon} on the projectivity of K\"ahler Moishezon spaces with $1$-rational singularities (Theorem~\ref{Namikawa}). In line with the assumptions made in Namikawa's theorem, the content of the subsequent sections is as follows: in Section \ref{section:functionfields}, we prove that the quotient $X(\mu)\hq G$ is Moishezon, in Section~\ref{section:singularitiesofmomentumquotients} that it has only $1$-rational singularities, using a refinement of Boutot's result~\cite{Boutot} on the rationality of quotient singularities. The proof of Theorem~\ref{1} is then given in Section~\ref{section:proofoftheorem1}. After a short discussion of basic properties of Chow quotients, Section~\ref{section:Chow} is devoted to the proof of Theorem~\ref{unstablepointsarealgebraic}. In Section~\ref{section:shifting}, Theorem~\ref{1} and Theorem~\ref{unstablepointsarealgebraic} are generalised to non-zero levels of the momentum map. From Section~\ref{sectionembedding} on we work in the setup of Theorem~\ref{2} and prove the analogues of Kodaira's Embedding Theorem and of Chow's Theorem. As a consequence of these results, we obtain Theorem~\ref{2} in Section~\ref{sectionalgebraicityofspaces}. In the final section, we discuss the question whether momentum map quotients of algebraic Hamiltonian $G$-varieties can be obtained as GIT-quotients, indicating open questions and directions for further study.
\subsection*{Acknowledgements}
The results presented in this note are contained in the author's dissertation \cite{Doktorarbeit}. He wants to thank his supervisor Prof.\ Dr.\ Peter Heinzner for inspiring discussions and mathematical advice. Furthermore, the author wants to thank the referee for helpful comments and remarks.
The author was supported by the Studienstiftung des deutschen Volkes and by SFB/TR 12 ''Symmetries and Universality of Mesoscopic Systems'' of the DFG.
\section{Quotients with respect to actions of complex reductive groups}
\subsection{Complex-reductive Lie groups and their representations}
Let $G$ be a complex reductive Lie group and let $K$ be a maximal compact subgroup of $G$. Then, the inclusion $\imath : K \hookrightarrow G$ is the universal complexification of $K$ in the sense of \cite{HochschildStructure}. In particular, every continuous group homomorphism from $K$ into a complex Lie group extends to $G$. Conversely, every compact Lie group has a universal complexification $\imath : K \hookrightarrow K^\C$, $K^\C$ is a complex reductive Lie group, and $\imath (K)$ is a maximal compact subgroup of $K^\C$. Every
complex reductive Lie group carries the unique structure of a linear-algebraic group.

A \emph{representation} of a Lie group $G$ will always mean a finite-dimensional complex continuous
re\-pre\-sen\-tation of $G$. If $\rho: G \to GL_\C(V)$ is a representation of $G$, then $V$ together with
the action of $G$ on $V$ induced by $\rho$ is called a \emph{$G$-module}. Given a linear algebraic
group $G$, a representation $\rho: G \to GL_\C(V)$ is called \emph{rational} if $\rho$ is a regular
map between the linear algebraic groups $G$ and $GL_\C(V)$. Any representation of a compact Lie
group $K$ extends to a rational representation of $K^\C$.
\subsection{Complex spaces and algebraic varieties}
In the following a \emph{complex space} refers to a reduced complex space with countable topology.
For a given complex space $X$ the structure sheaf will be denoted by $\mathscr{H}_X$. Analytic
subsets are assumed to be closed. By an \emph{algebraic variety} we mean an algebraic variety
defined over the field $\C$ of complex numbers. The structure sheaf of an algebraic variety
$X$ will be denoted by $\mathscr{O}_X$. Given an algebraic variety $X$, there exists a complex space canonically associated to $X$, see \cite{GAGA}. If it is necessary to distinguish
between $X$ and its associated complex space, we will denote
the latter by $X^h$. Any regular map $\phi: X \to Y$ of algebraic varieties is holomorphic with respect to the holomorphic structures on $X$ and $Y$. Again, if it is necessary to
distinguish between algebraic and holomorphic structures, we will write $\phi^h: X^h \to Y^h$ for
the induced map of complex spaces.
\subsection{Actions of Lie groups and analytic Hilbert quotients}
If $G$ is a real Lie group, then a \emph{complex $G$-space $Z$} is a
complex space with a real-analytic action $\alpha: G \times Z \to Z$ such that all the maps
$\alpha_g: Z \to Z$, $z \mapsto \alpha(g,z) = g \acts z$ are holomorphic. If $G$ is a complex Lie group, a
\emph{holomorphic $G$-space
$Z$} is a complex $G$-space such that the action map $\alpha: G \times Z \to Z$ is
holomorphic. A \emph{complex $G$-manifold} is a complex $G$-space without
singular points. In general, the set of singular points of a complex $G$-space
$X$ is a $G$-invariant analytic subset of $X$.

Let $G$ be a complex reductive Lie group and $X$ a holomorphic $G$-space. A
complex space $Y$ together with a $G$-invariant surjective holomorphic map $\pi: X \to Y$ is called
an \emph{analytic Hilbert quotient} of $X$ by the action of $G$ if
\vspace{-2.5mm}
\begin{enumerate}
 \item $\pi$ is a locally Stein map, and
 \item $(\pi_*\mathscr{H}_X)^G = \mathscr{H}_Y$ holds.
\end{enumerate}
\vspace{-1mm}
Here, \emph{locally Stein} means that there exists an open covering of $Y$ by open Stein subspaces
$U_\alpha$ such that $\pi^{-1}(U_\alpha)$ is a Stein subspace of $X$ for all $\alpha$; by
$(\pi_*\mathscr{H}_X)^G$ we denote the sheaf $U \mapsto \mathscr{H}_X(\pi^{-1}(U))^G = \{f \in
\mathscr{H}_X(\pi^{-1}(U)) \mid f \;\; G\text{-invariant}\}$, $U$ open in
$Y$.

An analytic Hilbert quotient of a holomorphic $G$-space $X$ is unique up to biholomorphism once it exists and we will denote it by $X\hq G$. It has the following properties (see \cite{SemistableQuotients}):
\begin{enumerate}
\item Given a $G$-invariant holomorphic map $\phi: X \to Z$ into a
complex
space $Z$, there exists a unique holomorphic map $\widebar \phi: Y \to Z$ such that
$\phi = \widebar \phi \circ \pi$.

\item For every Stein subspace $A$ of $X\hq G$ the inverse image $\pi^{-1}(A)$ is a Stein subspace
of $X$.
\item If $A_1$ and $A_2$ are $G$-invariant analytic subsets of $X$, then
$\pi(A_1) \cap \pi(A_2) = \pi(A_1 \cap A_2)$.
\item For a $G$-invariant analytic subset $A$ of $X$, the image $\pi(A)$ in
$Y = X\hq G$ is an analytic subset of $X\hq G$ and the restriction $\pi|_A: A \to \pi(A)$ is the
analytic Hilbert quotient for the action of $G$ on $A$.
\end{enumerate}
It follows that two points $x,x' \in X$ have the same image in $X\hq G$ if and only if
$\overline{G\acts x} \cap \overline{G\acts x'} \neq \emptyset$. For each $q \in X\hq
G$, the fibre $\pi^{-1}(q)$ contains a unique closed $G$-orbit $G\acts x$. The stabiliser $G_x$ of $x$ in $G$ is a complex reductive Lie group.

If $X$ is a Stein space, then the analytic Hilbert quotient exists and has the properties listed above (see \cite{HeinznerGIT} and \cite{Snow}).
\subsection{Algebraic Hilbert quotients}
Let $G$ be a complex reductive group endowed with its natural linear-algebraic structure. An
\emph{algebraic $G$-variety} is an algebraic variety $X$ together with an action of $G$ on $X$
such that the action map $G \times X \to X$ is regular. Basic examples of algebraic $G$-varieties are re\-pre\-sen\-ta\-tion spaces $V$ of rational
$G$-re\-pre\-sen\-ta\-tions $\rho: G \to GL_\C(V)$ of $G$.

An algebraic variety
$Y$ together with a $G$-invariant surjective regular map $\pi: X \to Y$ is called an \emph{algebraic
Hilbert quotient} of $X$ by the action of $G$ if
\vspace{-2.5mm}
\begin{enumerate}
 \item $\pi$ is affine, and
  \item $(\pi_*\mathscr{O}_X)^G = \mathscr{O}_Y$ holds.
\end{enumerate}

For a $G$-module or, more general, an affine algebraic $G$-variety $V$, the algebra $\C[V]^G$ of $G$-invariant polynomials is finitely generated and the map $V \to \Spec \left(\C[V]^G \right)$ is an algebraic Hilbert quotient.

If $X$ is an algebraic $G$-variety, the associated complex space $X^h$ is in a natural way a holomorphic $G$-space. If the algebraic Hilbert quotient $\pi: X \to X\hq G$ exists, the associated holomorphic map $\pi^h: X^h \to (X\hq G)^h$ is an analytic Hilbert quotient for the action of $G$ on $X^h$, cf.\ \cite{Lunaalgebraicanalytic}.
\section{Momentum map quotients}
Analytic Hilbert quotients arise naturally in the study of Hamiltonian actions on K\"ahlerian complex spaces. We recall the necessary definitions.
\subsection{K\"ahler structures}
A \emph{K\"ahler structure} on a complex space $X$ is
given by an open cover $(U_j)$ of $X$ and a family of strictly
plurisubharmonic functions $\rho _j: U_j
\rightarrow \R$ such that the differences $\rho_j -
\rho_k$ are pluriharmonic on $U_j \cap U_k$. Here \emph{strictly plurisubharmonic} means strictly plurisubharmonic with respect to perturbations cf.\ \cite{Extensionofsymplectic}). A \emph{pluriharmonic function} is a function which is locally the real part of a holomorphic function. A K\"ahler structure $\omega = \{\rho_j\}$ is called \emph{smooth} if all the $\rho_j$'s can be chosen as smooth functions. For smooth $X$, a smooth K\"ahler structure is the same as a smooth K\"ahler form which is given
locally by $\omega = \frac{i}{2} \partial \bar\partial \rho_j$. In this case we will not distinguish between
$\{\rho_j\}$ as defined above and the associated K\"ahler form.
\subsection{Momentum maps and analytic Hilbert quotients}\label{intromomentummapquotients}
Let $K$ be a Lie group with Lie algebra $\mathfrak{k}$. Let $X$ be a complex $K$-space endowed with a smooth $K$-invariant K\"ahler structure $\omega = \{\rho_j\}_{j}$. A \emph{momentum map} with respect to $\omega$ is a smooth map $\mu: X \to \mathfrak{k}^*$ which is $K$-equivariant with respect to
the coadjoint representation of $K$ on $\mathfrak{k}^*$ such that for every $K$-stable complex
submanifold $Y$ of $X$ and for every $\xi \in \mathfrak{k}$, we have
\[d\mu^\xi = \iota_{\xi_X}\omega_Y.\]
Here,  $\iota_{\xi_X}$ denotes contraction
with the vector field $\xi_X$ on $X$ that is induced by the $K$-action, $\omega_Y$ denotes the K\"ahler form induced on $Y$, and the function $\mu^\xi: X
\to \R$ is given by $\mu^\xi(x)=\mu(x)(\xi)$. We call the action of $K$ on a complex $K$-space with
$K$-invariant K\"ahler structure $\omega$ \emph{Hamiltonian} if the $K$-action admits a
momentum map with respect to $\omega$.

Let $K$ denote a compact Lie group and $G$ its complexification. Let $X$ be a
holomorphic $G$-space with $K$-invariant K\"ahler structure $\omega$. We call $X$ a
\emph{Hamiltonian $G$-space}, if the $K$-action is Hamilto\-ni\-an with respect to $\omega$. Given a
Hamiltonian $G$-space with momentum map $\mu: X \to \mathfrak{k}^*$, set
\[X(\mu) \definiere \{ x \in X \mid \overline{G\acts x} \cap \mu^{-1}(0) \neq \emptyset\}. \]We
call $X(\mu)$ the set of \emph{semistable points} with respect to $\mu$ and the $G$-action.
In general, if $X$ is a holomorphic $G$-space and $A$ is a subset of $X$, then we set $\mathcal{S}_G(A) \definiere \{ x \in X \mid \overline{G \acts x} \cap A \neq \emptyset\}$. With this definition, we have $X(\mu) = \mathcal{S}_G(\mu^{-1}(0))$. We collect the main results about Hamiltonian $G$-spaces in
\begin{thm}[\cite{ReductionOfHamiltonianSpaces},
\cite{Extensionofsymplectic}]\label{propertiesmomentumquotients}
Let $X$ be a Hamiltonian $G$-space. Then,
\begin{enumerate}
\item $X(\mu)$ is open and $G$-invariant, and the analytic Hilbert quotient $\pi: X(\mu) \to X(\mu)\hq G$ exists,
\item the inclusion $\mu^{-1}(0) \hookrightarrow X(\mu)$ induces a homeomorphism $\mu^{-1}(0)/K \simeq X(\mu)\hq G$,
\item the complex space $X(\mu)\hq G$ carries a K\"ahler structure that is smooth along a natural stratification of $X(\mu)\hq G$.
\end{enumerate}
\end{thm}
An \emph{algebraic Hamiltonian $G$-variety} is an algebraic $G$-variety $X$ such that the associated holomorphic $G$-space $X^h$ is a Hamiltonian $G$-space. In particular, $X^h$ is assumed to carry a $K$-invariant K\"ahler structure.
\begin{ex}\label{projectivemomentumexample}
Let $K$ be a compact Lie group with Lie algebra $\mathfrak{k}$ and let $G = K^\C$ be its
complexification. Let $\rho: G \to GL_\C(V)$ be a representation of $G$
such that the action of $K$ leaves a Hermitian inner product $\langle \cdot, \cdot \rangle$ on $V$ invariant. Endow $\P(V)$ with the Fubini-Study form induced by $\langle \cdot,\cdot \rangle $. Let
$X$ be a $G$-stable subvariety of $\P(V)$. Then, $X^{h}$ is K\"ahler with $K$-invariant K\"ahler
structure $\omega$ given by the restriction of the Fubini-Study form to $X$. In
fact, $X^h$ is a Hamiltonian $G$-space and a momentum map with respect to $\omega$ is given by
\[\mu^\xi ([v]) = \frac{2i \langle \rho_*(\xi) v,v \rangle }{\langle v,v  \rangle } \quad \quad \forall\, \xi \in \mathfrak{k}, \forall \,[v]
\in X.
\]
Here, $\rho_*: \Lie(G) \to \End _\C(V)$ is the Lie algebra homomorphism induced by $\rho$.
The set $X^h(\mu)$ of semistable points coincides with $X(V) = \{[v] \in X \mid \exists \,\text{ non-constant }f \in \C[V]^G: f(v) \neq 0\}$ and is therefore Zariski-open in $X$.
It follows that the analytic Hilbert quotient $X^h(\mu) \hq G$ is the complex space associated to
the projective algebraic variety $X(V)\hq G \definiere \Proj \left( \C[\text{Cone}(X)]^G \right)$; cf.\ \cite{KirwanCohomology}.
\end{ex}
\section{Function fields of momentum map quotients}\label{section:functionfields}
\subsection{Meromorphic functions and meromorphic graphs} We denote the sheaf of germs of \emph{meromorphic functions} on a complex space $X$ by $\mathscr{M}_X$. If $X$ is irreducible, $\mathscr{M}_X(X)= H^0 (X, \mathscr{M}_X)$ is a field, called the \emph{function field} of $X$. Let $f \in \mathscr{M}_X(X)$ be a meromorphic function on $X$. The
\emph{sheaf of denominators} of $f$ is the sheaf $\mathscr{D}(f)$ with stalks
$\mathscr{D}(f)_x = \{v_x \in \mathscr{H}_x : v_x f_x \in \mathscr{H}_x \}$. We
define the \emph{polar set} of $f$ to be the zero set of this sheaf and denote it by $P_f$. It is a nowhere dense analytic subset of $X$. The polar set $P_f$ is the smallest subset of $X$
such that $f$ is holomorphic on $X \setminus P_f$. We set $\dom (f) \definiere X \setminus P_f$ and
call it the \emph{domain of definition} of $f$.

For $f \in \mathscr{M}_X(X)$ we define $\Gamma ^0_f  \definiere \left\{(x, [f(x): 1] ) \mid x \in \dom (f)\right\} \subset X \times
\mathbb{P}_1(\C)$. The point set closure of $\Gamma_f^0$ in $X \times \P _1$ is called
the \emph{graph} of $f$. We denote it by $\Gamma_f$.

An analytic set $\Gamma \subset X \times  \P_1$ together with the canonical projection $\sigma: \Gamma \rightarrow X$ is called a \emph{holomorphic graph at
$p \in X$} if there exists an open neighbourhood $U$ of $p$ such that
\vspace{-2.5mm}
\begin{enumerate}
 \item $\sigma|_{\sigma^{-1}(U)}: \sigma^{-1}(U) \rightarrow U$ is biholomorphic,
 \item $\sigma^{-1}(U) \cap (U \times  \{[1:0]\}) = \emptyset$.
\end{enumerate}
Clearly, the graph of a holomorphic function is a holomorphic graph. An analytic subset $\Gamma \subset X \times \P_1$ with the canonical projection $\sigma: \Gamma \rightarrow X$ is called a \emph{meromorphic graph over $X$} if there exists an analytic subset $A \subset X$ such
that $A$ and $\sigma^{-1}(A)$ are nowhere dense and $\Gamma$ is a holomorphic graph outside $A$.

\begin{prop}[\cite{FischerMeromorphic}]\label{graphtofunction}
The graph $\Gamma_f$ of a meromorphic function on a complex
space is a meromorphic graph. Conversely, let $\Gamma \embeds X \times \P_1$ be a meromorphic graph over $X$. Then, there exists a uniquely determined meromorphic function $f \in \mathscr{M}_X(X)$ such
that $\Gamma = \Gamma_f$.
\end{prop}
If $X$ is a compact irreducible complex space, it is a classical result that the transcendence degree $\trdeg_\C \mathscr{M}_X(X)$ is less than or equal to the dimension of $X$, see e.g.\ \cite{CAS}. If $\trdeg_\C \mathscr{M}_X(X) = \dim X$, we call $X$ \emph{Moishezon}.

Let $X$ be a $G$-irreducible complex $G$-space. Consider the graph $\Gamma_f \subset X \times \P_1$ of a meromorphic function $f \in \mathscr{M}_X(X)$. The group $G$ acts on
$X \times \P_1$ by the $G$-action on the first factor. Given $g \in G$, we define a new
meromorphic graph $\Gamma_{g\acts f} \definiere g\acts \Gamma_f \subset X \times \P_1$.
This defines a meromorphic function $g\acts f$ on $X$ by Proposition \ref{graphtofunction}. In this way we obtain a group action on
$\mathscr{M}_X(X)$ by algebra automorphisms. A meromorphic function $f \in \mathscr{M}_X(X)$
is $G$-invariant if and only if its graph $\Gamma_f$ is a $G$-invariant subset of $X \times \P_1$. If $X$ is an algebraic $G$-variety, then $G$ acts naturally on the algebra $\C(X)$ of rational functions such that the inclusion $\C(X) \subset \mathscr{M}_{X^h}(X^h)$ is $G$-equivariant.
\subsection{Meromorphic functions on momentum map quotients}
Let $K$ be a compact Lie group and let $G=K^\C$ be its complexification. Let $X$ be a $G$-irreducible
algebraic Hamiltonian $G$-variety with momentum map $\mu: X \to \mathfrak{k}^*$. Let $Q =
X(\mu)\hq G$ denote the analytic Hilbert quotient of the set of semistable points with respect to
$\mu$. The $G$-irreducibility of $X$ implies that $X(\mu)$ is either empty or $G$-irreducible and dense in $X$ (see
\cite{Potentialsandconvexity}). It follows that $Q$ is irreducible.

We will use invariant rational functions on $X$ to produce meromorphic functions on $Q$. It should be noted that not every invariant rational function is obtained by pullback from the quotient, as the following example shows.
\begin{ex}
Consider the action of $G \definiere \C^* = (S^1)^\C$ on $X \definiere \C^2$ by scalar multiplication. The
action of $S ^1$ is Hamiltonian with respect to the standard K\"ahler form on $\C^2$ and, after identification of $\Lie (S^1)^*$ with $\R$, a momentum
map is given by $v \mapsto |v|^2$. We have $\mu^{-1}(0) = \{0\}$ and $X(\mu) = \C^2$. The analytic
Hilbert quotient $X(\mu)\hq G$ is a point. The rational function $f(z,w) = \frac{z}{w}$ on $X$ is $\C^*$-invariant. However, it is not
the pullback of a rational function from $X(\mu)\hq G = \{pt\}$ via the quotient map $\pi: \C^2 \to \{\mathrm{pt} \}$.
\end{ex}
If $X$ is a $G$-irreducible algebraic $G$-variety of an algebraic group $G$, let $m := \max _{x \in X} \{\dim G\acts x \}$. Then, $X_{\mathrm{reg}} \definiere \{x \in X \mid \dim G\acts x = m\}$ is a $G$-invariant Zariski-open subset of $X$.
\begin{lemma}\label{reductiontogeneric}
If $\mu^{-1}(0) \neq \emptyset$, there exists a $G$-irreducible subvariety $Y$ of $X$ such that $Y(\mu)\hq G \cong Q$ and $Y_{\mathrm{reg}} \cap \mu|_Y^{-1}(0) \neq \emptyset$.
\end{lemma}
\vspace{-3mm} \begin{proof}
We argue by induction on $\dim X$. If $\dim X = 0$, there is nothing to prove. In the general case, if $X_{\mathrm{reg}} \cap \mu^{-1}(0) \neq \emptyset$, we are done. So we can assume that $X_{\mathrm{reg}} \cap \mu^{-1}(0) = \emptyset$. Consider $Y \definiere X \setminus X_{\mathrm{reg}}$. We have $\mu^{-1}(0) \subset Y$ and therefore $\pi(Y(\mu)) = \pi(\mu^{-1}(0)) = Q$. Let $Y = \bigcup_{j=1}^m Y_j$ be the decomposition
of $Y$ into $G$-irreducible components. Then, we have $Y(\mu) = \bigcup_{j=1}^m Y(\mu) \cap Y_j =
\bigcup_{j=1}^m Y_j(\mu)$ and $Q = \bigcup_{j=1}^m \pi(Y_j(\mu))$. Since $Q$ is irreducible and each
$\pi(Y_j(\mu))$ is analytic in $Q$, there exists a $j_0 \in \{1,\dots, m\}$ such that $\pi(Y_{j_0}(\mu))
= Q$. It follows that $Y_{j_0}(\mu)\hq G \cong Q$. Since $\dim Y_{j_0} < \dim X$, we are done.
\end{proof}

\begin{lemma}
Assume that $X_{\mathrm{reg}} \cap \mu^{-1}(0) \neq \emptyset$. Then, the set $\pi(X_{\mathrm{reg}} \cap \mu^{-1}(0)) $ is an analytically Zariski-open dense subset of $Q$. We set $X^s \definiere
\pi^{-1}(\pi (X_{\mathrm{reg}} \cap \mu^{-1}(0))$. Then, $X^s$ is analytically
Zariski-open and dense in $X(\mu)$, we have $X^s = G\acts \mu^{-1}(0)$, and the restriction
$\pi|_{X^s}: X^s \rightarrow \pi(X^s)\subset Q$ is a geometric quotient for the $G$-action on
$X^s$, i.e., the fibres of $\pi|_{X^s}$ are orbits of the $G$-action.
\end{lemma}
\vspace{-3mm} \begin{proof}
Let $Y \definiere X \setminus X_{\mathrm{reg}}$. The intersection of $Y$ with $X(\mu)$ is analytic in
$X(\mu)$ and $G$-invariant. Since $\pi (Y \cap \mu^{-1}(0)) = \pi (Y \cap X(\mu)) $, the latter
being the image of a $G$-invariant analytic set in $X(\mu)$, $\pi(X_{\mathrm{reg}} \cap
\mu^{-1}(0))$
is analytically Zariski-open and dense in $Q$. It follows that $X^s$ is analytically Zariski-open and dense in
$X(\mu)$. Let $x \in X^s$. Then, the closure of $G\acts x$ in $X(\mu)$ contains a unique closed
$G$-orbit, say $G\acts y \subset \overline{G\acts x}$, for some $y \in X_{\mathrm{reg}} \cap \mu^{-1}(0)$.
Furthermore, since the action of $G$ on $X$ is algebraic, $G\acts x$ is Zariski-open in
$\overline{G\acts x}$. Since $y \in X_{\mathrm{reg}}$, this implies that $G\acts x = G
\acts y$. As a consequence, $G \acts x$ is closed in $X(\mu)$.
\end{proof}

The map $\pi \times \id_{\P_1}: X(\mu) \times \P_1 \rightarrow Q \times \P_1$ is an analytic
Hilbert quotient for the action of $G$ on $X(\mu) \times \P_1$ that is induced by the action of $G$ on
the first factor. Let $f \in \mathscr{M}_X(X(\mu))^G$ be a $G$-invariant meromorphic function on $X$
and let $\Gamma_f \subset
X(\mu) \times \P_1$ be its graph. The image $\widehat \Gamma \definiere \pi \times \id_{\P_1}(\Gamma_f)$ is an analytic subset
of $Q \times \P_1$. We set $\Phi = (\pi \times \id_{\P_1})|_{\Gamma_f}: \Gamma_f \to \widehat \Gamma
\cong \Gamma_f \hq G$.

\begin{prop}\label{invariantsgodown}Assume that $X_{\mathrm{reg}} \cap \mu^{-1}(0) \neq \emptyset$.
Let $f \in \mathscr{M}_X(X(\mu))^G$. Then, the analytic set $\widehat\Gamma \subset Q \times \P_1$ is a
meromorphic graph and defines a meromorphic function $\widehat f$ on $Q$ with $f = \pi^* (\widehat f)$.
\end{prop}
\vspace{-3mm} \begin{proof}
We summarise our setup in the following commutative diagram:
\begin{equation*}%\label{bigdiagram}
\begin{xymatrix}{
                            & X(\mu)\ar@/^7pc/[ddd]^{\pi}                            & \\
 X(\mu) \times \mathbb{P}_1\ar[d]_{\pi \times \id_{\P_1}} &\ar@{_{(}->}[l]
\Gamma_f\ar[u]_{\widetilde{\pr}_1}\ar[r]\ar[d]^{\Phi}&
\mathbb{P}_1 \\
Q \times \mathbb{P}_1 & \ar@{_{(}->}[l] \widehat \Gamma_{\ } \ar[r]\ar[d]^{\pr_1} & \mathbb{P}_1\\
& Q &
}
\end{xymatrix}
\end{equation*}
Since $f \in \mathscr{M}_X(X(\mu))^G$, the polar set $P_f$ of $f$ is $G$-invariant.
Let $V \definiere X^s \setminus P_f$. Since $X^s$ is Zariski-open and dense in $X(\mu)$ and
since $P_f$ is a nowhere dense analytic subset of $X(\mu)$, $V$ is analytically Zariski-open and
dense in $X(\mu)$. Since $\pi|_{X^s}: X^s \to \pi(X^s)$ is a geometric quotient, and since $P_f$ is
$G$-invariant, $V$ is $\pi$-saturated. Thus, $\pi(V)$ is analytically Zariski-open
and dense in $Q$. The subset $\pi(V) \subset Q$ is isomorphic to $V\hq G$ and we have $f\in \mathscr{H}_X(V)^G \cong
\mathscr{H}_Q(\pi(V))$. Therefore, there exists a unique holomorphic function $\widehat f \in
\mathscr{H}_Q(\pi(V))$ such that $f = \widehat f \circ \pi$. We have $\Gamma_{\widehat f} = \widehat \Gamma|_{\pi(V)} \subset \pi(V) \times \C$. It follows that
$\widehat \Gamma$ is a holomorphic graph over $\pi(V)$.

The set $A \definiere X(\mu) \setminus V$ is $\pi$-saturated and analytic in $X(\mu)$. Hence,
$\pi(A) = Q \setminus \pi(V)$ is a nowhere dense analytic subset of $Q$. Since $\Gamma_f$ is a
meromorphic graph over $X(\mu)$ and since $A$ is nowhere dense in $X(\mu)$, $\widetilde{\pr}_1^{-1}(A) = \Phi^{-1}_{\ }(\pr_1^{-1}(\pi(A)))$
is nowhere dense in $\Gamma_f$. It follows that $\pr_1^{-1}(\pi(A))$ is nowhere dense in $\widehat
\Gamma$.

We have hereby shown that $\widehat \Gamma$ is a
meromorphic graph in
$Q\times  \P_1$. By Proposition \ref{graphtofunction}, it defines a
meromorphic function $\widehat f$ on $Q$. It follows from the construction that $f = \pi^*(\widehat f)$
holds.
\end{proof}
\begin{lemma}\label{dimensionofquotient}
Assume $X_{\mathrm{reg}}\cap \mu^{-1}(0) \neq \emptyset$. Then, $\dim Q = \dim X -m,$ where $m = \max_{x\in X}\{\dim G\acts x\}$.
\end{lemma}
\vspace{-3mm} \begin{proof}
Since $X^s \subset X_{\mathrm{reg}}$ and $\pi|_{X^s}: X^s \rightarrow \pi(X^s)$ is a geometric
quotient, we have $\dim \pi^{-1} (q) =  m$ for every $q \in \pi(X^s)$. The set $\pi(X^s)$ is
Zariski-open and dense in $Q$. It follows that $\dim Q = \dim X - m$ by the theorem on the fibre
dimension of holomorphic maps (Satz 15 in \cite{RemmertAbbildungen}).
\end{proof}
\begin{thm}\label{momentumquotientsMoishezon}
Let $K$ be a compact Lie group and $G=K^\C$ its complexification. Let $X$ be a $G$-irreducible
algebraic Hamiltonian $G$-variety with momentum map
$\mu: X \rightarrow \k ^*$. Let $Q= X(\mu) \hq G$ denote the analytic Hilbert
quotient. Then, $\trdeg_\C\mathscr{M}_Q(Q)
\geq \dim Q$. In particular, if $Q$ is compact, then it is Moishezon.
\end{thm}
\vspace{-3mm} \begin{proof}
By Lemma \ref{reductiontogeneric}, we can assume without loss of generality that $X_{\mathrm{reg}}\cap \mu^{-1}(0) \neq \emptyset$. By a Theorem of Rosenlicht \cite{Rosenlicht2}, we have $\trdeg _\C \C(X)^G = \dim X - m$, where $m$ is defined as in Lemma \ref{dimensionofquotient}.
Furthermore, by Proposition \ref{invariantsgodown}, we have defined a field
homomorphism $\mathscr{M}_X(X(\mu))^G \hookrightarrow \mathscr{M}_Q(Q)$. Hence, we get the following
chain of inequalities
\begin{align*}
\dim Q     &=   \dim X -m &&\text{by Lemma \ref{dimensionofquotient}}\\
           &= \trdeg_\C \C(X)^G   &&\text{by
           Rosenlicht's theorem}\\
           &\leq \trdeg_\C \mathscr{M}_X\left(X (\mu) \right)^G
           &&  \\
           &\leq \trdeg_\C\mathscr{M}_Q(Q) &&\text{by Prop. \ref{invariantsgodown}}.\qedhere
\end{align*}
\end{proof}
\section{Singularities of momentum map quotients}\label{section:singularitiesofmomentumquotients}
\subsection{$1$-rational singularities}
A normal algebraic variety or normal complex space $X$ is said to have \emph{$\mathit{1}$-rational
singularities}, if for every resolution $f: \widetilde X \rightarrow X$ of $X$, we have $R^1f_*\mathscr O _{\widetilde X} = 0$ or $R^1f_* \mathscr{H}_{\widetilde X} = 0$, respectively. Here, $R^1f_*\mathscr{O}_{\widetilde X}$ denotes the sheaf associated to the presheaf $U \mapsto H^1(\pi^{-1}(U), \mathscr{O}_{\widetilde X})$, and $R^1f_* \mathscr{H}_{\widetilde X}$ is defined analogously.

As in the case of rational singularities (where one requires the vanishing of all higher direct image sheaves of a resolution) using \cite{Hironaka}, \S0.7, Cor.\ 2, it can be shown that the condition $R^1f_*\mathscr O _{\widetilde X} = 0$ is independent of the chosen resolution.
Similarly to \cite{KollarMori}, Corollary 5.11, one proves that an algebraic variety $X$ has
$1$-rational singularities if and only if $X^{h}$ has $1$-rational singularities.
%If an algebraic group $G$ acts algebraically on $X$, the support of $R^1f_*\mathscr O _{\widetilde X}$ is a $G$-invariant algebraic subset of $X$.
\begin{rem}
The class of varieties with $1$-rational singularities is a strictly bigger than the class of varieties with rational singularities: the vertex of the cone over a smooth quartic surface in $\P_3(\C)$ is a $1$-rational singularity that is not rational.
\end{rem}

\subsection{Singularities of algebraic Hilbert quotients}
In our study of singularities of momentum map quotients we use the following result about the singularities of algebraic Hilbert quotients.
\begin{thm}[\cite{GrebSingularities}]\label{rationalsingularities}
Let $G$ be a complex reductive Lie group and let $X$ be an algebraic $G$-variety such that the
algebraic Hilbert quotient $\pi: X \to X\hq G$ exists. If $X$ has $\mathit{1}$-rational
singularities, then $X\hq G$ has $\mathit{1}$-rational singularities.
\end{thm}

\begin{rem}
The proof closely follows Boutot's proof \cite{Boutot} that algebraic Hilbert quotients of varieties with rational singularities have rational singularities. The main point to check is that in Boutot's arguments it is possible to treat the different cohomology degrees separately .
\end{rem}
\subsection{Local linearisation by the Slice Theorem}
Let $K$ be a compact Lie group and let $X$ be an algebraic Hamiltonian $K^\C$-variety with momentum map $\mu: X \to \mathfrak{k}^*$. Our goal is
to show that if $X$ has $1$-rational singularities, then the analytic Hilbert quotient $X(\mu)\hq G$ has $1$-rational singularities. This section is devoted to the reduction of this problem to the equivalent question for algebraic Hilbert quotient with the help of a slice theorem. We start with two elementary lemmata.

\begin{lemma}\label{affineinprojective}
Let $G$ be a complex reductive Lie group and $V$ a $G$-module. Consider the
induced action of $G$ on $\P(V)$ and let $[v] \in \P(V)^G$. If $L \definiere \C v \subset V$ and $L^\perp$ is a $G$-stable complementary subspace to $L$ in $V$, then $W \definiere \P(V) \setminus \P(L^\perp)$ is a $G$-invariant affine open neighbourhood of $[v]$ in $\P(V)$. Furthermore, $W$ is $G$-equivariantly isomorphic to the $G$-module $L^\perp \otimes L^{-1}$, where $L^{-1}$ is the dual of the $G$-module $L$.

%If $[v] \in \P(V)^G$, there exists a $G$-invariant
%hyperplane $H\subset \P(V)$ not containing $[v]$. The set $W
 %\definiere \P(V)\setminus H$ is $G$-equivariantly isomorphic to the
%$G$-module $L^\perp \otimes L^{-1}$, where $L^\perp$ is a $G$-stable
%complementary subspace to $L \definiere \C v \subset V$, and $L^{-1}$ is the dual
%of the $G$-module $L$.
\end{lemma}

\begin{lemma}\label{invariantaffine}
Let $G$ be a connected complex reductive Lie group and let $X$ be a normal
algebraic $G$-variety. Let $x \in X$ have reductive stabiliser $G_x$
in $G$. Then, there exists an affine $G_x$-invariant open
neighbourhood of $x$ in $X$.
\end{lemma}
\vspace{-3mm} \begin{proof}[Proof of Lemma~\ref{invariantaffine}]
Let $x \in X$ have reductive stabiliser $G_x$. Since $G$ is connected and $X$ is normal, by a result of Sumihiro \cite{completion}, there
exists a $G$-invariant quasi-projective open neighbourhood $U$ of $x$ in
$X$. Furthermore, there exists a $G$-linearised ample line bundle on
$U$, i.e., we can assume that $X=U$ is a $G$-invariant locally closed
subset of $\P(V)$, where $V$ is a $G$-module. Clearly, $V$ is also a
$G_x$-module and $x \in \P(V)^{G_x}$. Lemma \ref{affineinprojective}
implies that there exists a $G_x$-invariant affine open neighbourhood $W$ of
$x$ in $\overline X \subset \P(V)$. Let $C\definiere W \setminus X$ and let $\pi: W \rightarrow W\hq
G_x$ denote the algebraic Hilbert quotient. Then, $\{x\}$ and $C$
are disjoint $G_x$-invariant algebraic subsets of $W$ and hence
$\pi(x) \notin \pi(C) \subset W\hq G_x$. It follows that there exists
an affine open neighbourhood $V$ of $\pi(x)$ in $(W \hq G_x) \setminus
\pi (C)$. Since $\pi$ is an affine map, the inverse image $\pi^{-1} (V)$ is a $G_x$-invariant affine
open subset of $X$ containing $x$.
\end{proof}
Using the two previous lemmata, we now adapt the proof of the
holomorphic Slice Theorem (see \cite{HeinznerGIT}, \cite{ReductionOfHamiltonianSpaces}, \cite{CartanDecomposition} and \cite{SjamaarSlices}) to our algebraic situation.
\begin{thm}[Slice Theorem, algebraic version]\label{algebraicslice}
Let $K$ be a connected compact Lie group and $G = K^\C$ its complexification. Let $X$ be a
normal algebraic Hamiltonian $G$-variety with momentum map $\mu: X \to \mathfrak{k}^*$. Let $\pi:
X(\mu) \rightarrow X(\mu)\hq G$ be the analytic Hilbert quotient. Then, every point
in $\mu^{-1}(0)$ has a $\pi$-saturated open neighbourhood $U$ in $X(\mu)$
that is $G$-equivariantly biholomorphic to a saturated open
subset of an affine $G$-variety.
\end{thm}
Here, a \emph{saturated subset of an affine $G$-variety $Y$} is a subset of $Y$ that is saturated
with respect to the algebraic Hilbert quotient $\pi_Y:Y \to Y\hq G$.
\begin{rem}
By the holomorphic Slice Theorem every point
in $\mu^{-1}(0)$ has a $\pi$-saturated open neighbourhood $U$ in $X(\mu)$
that is $G$-equivariantly biholomorphic to a $G$-invariant analytic subset $A$ of a saturated open subset of a $G$-module. The crucial point of Theorem \ref{algebraicslice} is that in our situation we can choose $A$ to be the intersection of an affine $G$-invariant subvariety with a saturated open subset of a $G$-module. If $X$ is smooth, this follows directly from the holomorphic Slice Theorem, since in this case, $A$ can be chosen to be a saturated open subset of a $G$-module.
\end{rem}
\vspace{-3mm} \begin{proof}[Proof of Theorem \ref{algebraicslice}]
As in the proof of Lemma \ref{invariantaffine}, we can assume that $X$ is a $G$-stable locally closed subvariety of $\P(V)$, where $V$ is a $G$-module. Let $x =[v] \in \mu^{-1}(0) \subset X \subset
\P(V)$ and set $H \definiere G_x$. Since $x \in \mu^{-1}(0)$, its stabiliser group $H$ is reductive (see
\cite{HeinznerGIT}, Section 4.2). Using Lemma \ref{affineinprojective} it follows that $V$ is isomorphic to $L \oplus L^\perp$ as an $H$-module, where $L = \C v$, that there exists an $H$-invariant affine open neighbourhood $U\definiere \P(V) \setminus \P(L^\perp)$ of $x$ in $\P(V)$ as well as an $H$-equivariant isomorphism $\psi: U \to L^\perp \otimes L^{-1}$ with $\psi(x) = 0$. Let $B \definiere d\psi (x)\left(T_x (G\acts x)\right)$. Since $B$ is an
$H$-stable linear subspace of $L^\perp \otimes L^{-1}$, there exists an $H$-stable linear subspace
$N$ of $L^\perp \otimes L^{-1}$ such that $L^\perp\otimes L^{-1} = B \oplus N$. Then,
$\overline{\psi^{-1}(N)} \subset \P(V)$ is an $H$-invariant linear subspace of $\P(V)$, i.e., there exists an $H$-invariant linear subspace $W$ in $V$ such that $\overline{\psi^{-1}(N)} = \P(W)$. We
have $T_x(G\acts x) \oplus T_x(\P(W)) = T_x (\P(V))$,
by construction.

It follows that the map $\widetilde \varphi : G \times_H \P(W) \to \P(V), [g,y] \mapsto g\acts y$ is locally biholomorphic at $[e,x]$. Set $S \definiere X \cap \P(W)$. Since $\widetilde \varphi ^{-1} (X) =
G \times_H S$ by construction, we see that
$\varphi: G \times_H S \to X,\;
          [g,y]     \mapsto g\acts y$
is locally biholomorphic at $[e,x]$. Intersecting $S$ with a $G_x$-invariant affine open
neighbourhood of $x$ in $X$, we may assume that $S$ is affine. Furthermore, the restriction of $\varphi$ to $G\acts [e,x]$ is biholomorphic onto $G\acts x$. It follows from the
holomorphic version of Luna's fundamental lemma (see Lemma 14.3 and Remark 14.4 in
\cite{CartanDecomposition}; also see Lemma \ref{fundamental} of this note) that there exists an open $G$-stable
neighbourhood $U'$
of $[e,x]$ which is mapped biholomorphically onto the open neighbourhood $\varphi(U')$
of $G\acts x$ in $X$ by $\varphi$.

Since $G\times_H S$ is affine, there exists an algebraic Hilbert quotient
$\pi_S: G \times_H S \rightarrow (G\times_H S)\hq G$. By shrinking $U'$ if necessary, we can assume
that it is $\pi_S$-saturated. Since $X(\mu)$ is open in $X$, we may assume that $\varphi(U')
\subset X(\mu)$. Let $C\definiere X(\mu) \setminus \varphi(U')$. This is a closed $G$-invariant
subset of $X(\mu)$ not containing the closed orbit $G\acts x \subset X(\mu)$. Hence, $\pi(x) \notin
\pi (C)$. Replace $U'$ by $\widetilde U \definiere \varphi^{-1}(\pi^{-1}(X(\mu)\hq G \setminus
\pi(C)))\subset U'$. The set $\widetilde{U}$ is $\pi_S$-saturated. It follows that $\varphi(\widetilde{U})$ is a $\pi$-saturated open
neighbourhood of $x$ in $X(\mu)$ which is $G$-equivariantly biholomorphic to a saturated open subset
of an affine $G$-variety.
\end{proof}
\subsection{Singularities of momentum map quotients}
We are now in the position to prove
\begin{thm}\label{momentumquotientsrational}
Let $K$ be a connected compact Lie group and $G= K^\C$ its complexification. Let $X$ be an
irreducible algebraic Hamiltonian $G$-variety with momentum map $\mu: X \to \mathfrak{k}^*$. Assume
that $X$ has $\mathit{1}$-rational singularities. Then, the analytic Hilbert quotient $X(\mu)\hq G$ has
$\mathit{1}$-rational singularities.
\end{thm}
\begin{rem}
Theorem \ref{momentumquotientsrational} follows from Boutot's result \cite{Boutot}, if we additionally assume that $X$ is smooth. To prove it in full generality, we use the analysis done in the preceding sections.
\end{rem}
\vspace{-3mm}
\begin{proof}[Proof of Theorem \ref{momentumquotientsrational}]
Let $x \in \mu^{-1}(0)$. By the Slice Theorem, Theorem \ref{algebraicslice}, there exists an open
saturated neigh\-bour\-hood $U$ of $x$ in $X(\mu)$ and a
$G$-equivariant biholomorphic map $\varphi: U \to \widetilde{U}$ to a saturated
open
subset $\widetilde U$ of an affine $G$-variety $A$. The
algebraic Hilbert quotient of $A$ will be denoted by $\pi_A: A \to A\hq G$. Let $f:
\widehat A \rightarrow A$ be a $G$-equivariant resolution of $A$ and set $B \definiere \supp(R^1f_*
\mathscr{O}_{\widehat A})$. This is a $G$-invariant algebraic subvariety of $A$, and we have $\widetilde{U} \subset A \setminus B$. Since the $G$-orbit through $\varphi (x)$ is closed in $A$, we have $\pi(\varphi(x)) \notin \pi(B)$. Therefore, there exists an affine $\pi_A$-saturated open neighbourhood $\widetilde{A}$ of
$\varphi (x)$ in $A \setminus B$ such that $\widetilde{U} \subset  \widetilde{A}$.

Let $y \in X(\mu)\hq G$. Applying the considerations above to the unique closed orbit $G\acts x$ in $\pi^{-1}(y)$, there exists a neighbourhood $W$ of $y
\in X(\mu)\hq G$ such
that $\pi^{-1}(W)$ is $G$-equivariantly biholomorphic to a saturated
subset $\widetilde{U} \subset A$ in an affine $G$-variety $A$ with
$1$-rational singularities. Consequently, $W$ is biholomorphic to
$(\widetilde{U}\hq G)^h \subset (A\hq G)^h$, which has $1$-rational singularities by Theorem
\ref{rationalsingularities}. This concludes the proof.
\end{proof}
\section{Projectivity of momentum map quotients: Proof of Theorem 1}\label{section:proofoftheorem1}
In the proof of Theorem~\ref{1} given in this section, we use the following projectivity criterion to conclude that momentum map quotients of algebraic Hamiltonian $G$-varieties are projective:
\begin{thm}[\cite{ProjectivityofMoishezon}]\label{Namikawa}
Let $Z$ be a Moishezon space with $\mathit{1}$-rational singularities. If $Z$ admits a
smooth K\"ahler structure, then it is a projective algebraic
variety.
\end{thm}
\subsection{Proof of Theorem~\ref{1}}
First, let us additionally assume that $G$ is connected. It follows that $X$ is irreducible. The analytic Hilbert quotient $X(\mu)\hq G$ carries a continuous
K\"ahler structure by Theorem \ref{propertiesmomentumquotients}. Using a theorem of
Varouchas (\cite{Varouchas2}, Theorem 1) we conclude that there exists a smooth K\"ahler structure on $X(\mu)\hq G$.
Owing to Theorem \ref{momentumquotientsMoishezon} and Theorem
\ref{momentumquotientsrational}, $X(\mu)\hq G$ is a Moishezon space with $1$-rational
singularities. Consequently, Theorem \ref{Namikawa}
applies, and $X(\mu)\hq G$ is projective algebraic.

In the general case, let $K_0$ denote the connected component of the identity in $K$. Then, the complexification $(K_0)^\C$ of $K_0$ is equal to the
connected component of the identity in $G$, which we denote by $G_0$. Since the momentum map $\mu$ depends only on the infinitesimal action of $K$, it is also a momentum map for the $K_0$-action, which we will call
$\mu_{K_0}$. We have $\mu^{-1}_{}(0) = \mu^{-1}_{K_0}(0)$ and $X(\mu_{K_0}) = \mathcal{S}_{G_0}(\mu_{K_0}^{-1}(0))  \subset \mathcal{S}_G(\mu^{-1}(0))=X(\mu)$. Conversely, it follows from the Exhaustion Theorem of \cite{Potentialsandconvexity}, that $X(\mu) \subset
X(\mu_{K_0})$. Hence, we have $X(\mu) = X(\mu_{K_0})$.

Since $X$ is assumed to have $1$-rational singularities, it is normal, and consequently, its irreducible components are disjoint and $G_0$-stable. If $X
= \bigcup_j X_j$ is the decomposition of $X$ into irreducible components, then the decomposition of $X(\mu)$ into irreducible components is given by $X(\mu) = \bigcup_j X_j(\mu_{K_0}|_{X_j})$.
Setting $X(\mu)_j \definiere X_j(\mu_{K_0}|_{X_j})$, we see that the restriction of the analytic Hilbert quotient $\pi_{G_0}: X(\mu) \to X(\mu)\hq G_0$ to $X(\mu)_j$ is the analytic Hilbert
quotient for the action of $G_0$ on $X(\mu)_j$ and that $X(\mu)\hq G_0 =\bigcup_j \pi_{G_0}
(X(\mu)_j) = \bigcup_j X(\mu)_j\hq G_0$ is the decomposition of $X(\mu)\hq G_0$ into
irreducible components. The finite group $\Gamma \definiere G/G_0$ of
connected components of $G$ acts on $X(\mu)\hq G_0$, and we have the
following commutative diagram
\enlargethispage{2mm}
\[\begin{xymatrix}{
X(\mu)\ar[r]^\pi\ar[d]_{\pi_{G_0}} & X(\mu)\hq G \\
X(\mu)\hq G_0\ar[ru]_{\pi_\Gamma},
}
\end{xymatrix}
\]
where $\pi_\Gamma$ is the analytic Hilbert quotient for the
action of $\Gamma$ on $X(\mu)\hq G_0$.

Since $X(\mu)\hq G$ is compact by assumption, $X(\mu)\hq G_0$ is also
compact. Let $V$ be any of the irreducible
components of $X(\mu)\hq G_0$. Then, $V = X_j(\mu_{K_0}|_{X_j})$ for some $j$, as we have seen above. By the first part of the proof (the case $G$ connected), it follows that $V$ is projective algebraic. Since the irreducible components of $X(\mu)\hq G_0$ are disjoint, the projectivity of $X(\mu)\hq G_0$ follows.

As $X(\mu)\hq G_0$ is projective, the finite group $\Gamma$ acts algebraically on
$X(\mu)\hq G_0$.  It is a classical result that the algebraic Hilbert quotient $Y$ for the action of the finite group $\Gamma$ on the projective algebraic variety $X(\mu)\hq G_0$ exists as a projective algebraic
variety. Since $Y^h$ is the analytic Hilbert quotient for the induced $\Gamma$-action on $(X(\mu)\hq G_0)^h$, it is
biholomorphic to $X(\mu)\hq G$. Hence, $X(\mu)\hq G$ is a projective
algebraic complex space and the claim is shown.
\qed
\begin{rems}
a) Applying Theorem \ref{rationalsingularities} to $\pi_\Gamma: X(\mu)\hq G_0 \to X(\mu)\hq G$, we see that $X(\mu)\hq G$ also has $1$-rational singularities.

b) The assumption that $\mu^{-1}(0)$ is compact is necessary to obtain an algebraic structure on the analytic Hilbert quotient $X(\mu)\hq G$, cf. Example \ref{noncompactcounterexample}.

c) The assumption that $X$ is $G$-irreducible is not necessary. In general, apply Theorem~\ref{1} to the $G$-irreducible components of $X$ to conclude that every irreducible component of $X(\mu)\hq G$ is projective. Since these components are disjoint, the projectivity of $X(\mu)\hq G$ follows.
\end{rems}
\begin{ex}\label{noncompactcounterexample}
Consider the affine algebraic variety $X \definiere \C^* \times \C$ endowed with the $\C^*$-action on the first factor and with the K\"ahler structure $\omega = 2i \partial \bar \partial \rho$, where $\rho(t,z) = |t|^2|e^z-1|^2 + |t^{-1}|^2$ is smooth, $S^1$-invariant, and strictly plurisubharmonic. After identifying $\Lie (S^1)^*$ with $\R$, the associated momentum map for the $S^1$-action on $X$ is
\[\mu(t,z) = |t|^2|e^z-1|^2 - |t^{-1}|^2.\]
Note that $X(\mu) = X \setminus (\C^* \times 2\pi i\Z)$ is not algebraically Zariski-open in $X$. We obtain $X(\mu)\hq \C^* = X(\mu)/\C^* = \C \setminus 2\pi i \Z =: Q$. We claim that $Q$ is not the complex manifold associated to any algebraic variety. Aiming for a contradiction, we suppose that $Q = C^h$ for some smooth algebraic variety $C$. Then, $C$ is a non-complete algebraic curve. It follows that $C$ is affine (see \cite{HartshorneAmple}, Section II.4), and that there exist a regular open embedding $\varphi: C \hookrightarrow \overline C$ of $C$ into a smooth projective curve $\overline C$. Hence, there exist finitely many points $c_1, \dots, c_m$ in $\overline C$ such that $C= \overline C \setminus \{c_1, \dots, c_m\}$. If $\overline C$ has genus $g$, it follows that the fundamental group $\pi_1(Q)$ is freely generated by $2g + m -1$ generators, a contradiction.
\end{ex}
\begin{rem}
Theorem \ref{1} applies in particular if the momentum map $\mu$ under consideration is proper. Note however that the properness assumption on $\mu$ is rather strong, as the following result (see \cite{Doktorarbeit}, Section 2.6) indicates:
\begin{prop}
Let $X$ be a holomorphic $\C^*$-principal bundle over a compact complex
ma\-ni\-fold $B$. Then there exists a K\"ahler structure $\omega$ on $X$ such that
the induced $S^1$-action on $X$ is Hamiltonian with proper momentum map if and
only if $B$ is K\"ahler and $X$ is topologically trivial.
\end{prop}
\end{rem}
\section{Chow quotients and sets of semistable points}\label{section:Chow}
If $X$ is a Hamiltonian $G$-space for the complex reductive group
$G=K^\C$ with momentum map $\mu:~X \to \mathfrak{k}^*$, we have seen
in Section \ref{intromomentummapquotients} and used it throughout our study that $X(\mu)$ is an
open subset of $X$. From the point of view of complex analytic
geometry, it is natural to ask if the
set of \emph{unstable points} $X \setminus X(\mu)$ is an analytic
subset of $X$. The following example (see \cite{ReductionOfHamiltonianSpaces}) shows that this is
not true in general.
\begin{ex}\label{counterexampledense}
Let $X= \C^2 \setminus \{(x,y)\in \C^2 \mid y = 0, \Re\mathfrak{e} (x)\leq 0 \}$
with the standard K\"ahler form. Let $G = \C^*$ act on $X$ by
multiplication in the second variable. The $S^1$-action is Hamiltonian
with momentum map $\mu: X \to \Lie(S^1)^* \cong \R$, $(x,y) \mapsto |y|^2$. Hence, $\mu^{-1}(0) = \{(x,y) \in X \mid y = 0\}$, and the set of semistable points
$X(\mu) = \{(x,y) \in X \mid \Re\mathfrak{e} (x) > 0 \}$ is not analytically Zariski-open in $X$.
\end{ex}
In the remainder of this section we will prove Theorem~\ref{unstablepointsarealgebraic} which asserts that the set of semistable points of an algebraic Hamiltonian $G$-variety with compact quotient is algebraically Zariski-open.
\subsection{The Chow quotient of a projective algebraic $G$-variety}\label{Chowquotient}
Let $X$ be an irreducible projective algebraic $G$-variety for a connected complex reductive Lie
group $G$. For
$m, d \in \N$, let $\mathscr{C}_{m,d}(X)$ be the Chow variety of cycles of dimension $m$ and degree
$d$ in $X$. This is a projective algebraic component of the Barlet cycle space (see
\cite{BarletCycles}). For a
cycle $C \in \mathscr{C}_{m,d}(X)$ let $|C|$ denote the support of $C$. The following result provides a crucial tool for our study:
\begin{thm}[\cite{KapranovChow}]\label{Chow}
Let $X$ be an irreducible projective algebraic $G$-variety. Then, there exist natural numbers $m,d
\in \N$, a $G$-invariant rational map $\varphi: X \dasharrow \mathscr{C}_{m,d}(X)$, and a non-empty
$G$-invariant Zariski-open subset $U_0 \subset \dom(\varphi)$ such that for all $u \in U_0$, we have $|\varphi (u)| = \overline{G\acts u}$.
\end{thm}
We will denote the image of $\varphi$ in $\mathscr{C}_{m,d}(X)$ by $\mathscr{C}_G(X)$, and we will
call $\varphi: X \dasharrow \mathscr{C}_G(X)$ the \emph{Chow quotient of $X$ by $G$.} For the following discussion recall that for a subset $A$ of a $G$-space $X$, we have defined the saturation $\mathcal{S}_G(A)$ to be $\{ x \in X \mid \overline{G \acts x} \cap A \neq \emptyset\}$.
\begin{lemma}\label{saturationconstructible}
Let $A$ be a $G$-invariant irreducible locally Zariski-closed subset of $X$. Then, $\mathcal{S}_G(A)
\cap U_0$ is constructible.
\end{lemma}
\vspace{-3mm} \begin{proof}
We define $I_A \definiere \{C \in \mathscr{C}_G(X) \mid |C| \cap A \neq \emptyset \}$. We claim that
$I_A$ is a constructible subset of $\mathscr{C}_G (X)$. Consider the universal family $\mathscr{X} \definiere \{(C, x) \in \mathscr{C}_G(X) \times
X \mid x \in |C|\}$ over $\mathscr{C}_G(X)$. This is an
algebraic subvariety of $\mathscr{C}_ G(X) \times X$. Let $p_\mathscr{C} : \mathscr{X} \to
\mathscr{C}_G(X)$ and $p_X: \mathscr{X} \to X$ be the canonical projection maps. We have $I_A =
p_{\mathscr{C}}^{} (p_X^{-1}(A))$. Since $A$ is locally Zariski-closed, $I_A$ is a constructible subset of
$\mathscr{C}_G(X)$. It follows that $U_0 \cap \varphi^{-1}(I_A) = \{x \in U_0 \mid \overline{G\acts
x} \cap A \neq \emptyset\}$ is constructible in $U_0$.
\end{proof}
\subsection{Sets of semistable points: proof of Theorem~\ref{unstablepointsarealgebraic}}
The following result is the crucial step in the proof of Theorem~\ref{unstablepointsarealgebraic}.
\begin{prop}\label{prop1}
Let $G$ be a connected complex reductive Lie group and let $X$ be an irreducible algebraic
Hamiltonian $G$-variety  with momentum map $\mu: X \to \mathfrak{k}^*$. Let $X_{\mathrm{reg}}$ be the union of generic $G$-orbits in $X$. Assume that $X_{\mathrm{reg}} \cap
\mu^{-1}(0) \neq \emptyset$ and that $X(\mu)\hq G$ is projective algebraic. Then, $X(\mu)$ contains
a non-empty Zariski-open $G$-invariant subset $A$ of $X$ consisting of $G$-orbits that
are closed in $X(\mu)$.
\end{prop}
\vspace{-3mm} \begin{proof}
By a theorem of Rosenlicht \cite{Rosenlicht2} there exists a non-empty Zariski-open $G$-invariant subset $U_R$ of $X$ with geometric quotient $p_R: U_R \to Q_R$ such that $\C(Q_R) = \C(X)^G$. Note that $U_R$ intersects $\mu^{-1}(0)$ and that $Q_R$ is birational to $X(\mu) \hq G$. Let $\psi=\varphi^{-1}: Q_R \dasharrow X(\mu)\hq G$. Without loss of generality, we can assume that
the set where $\psi$ is an isomorphism onto its image is equal to $Q_R$. Hence, $\psi: Q_R
\hookrightarrow X(\mu)\hq G$ is an open embedding. Setting $A \definiere X(\mu) \setminus
\mathcal{S}_G(U_R^c \cap X(\mu)) \subset U_R$, we obtain the following commutative diagram:
\[\begin{xymatrix}{
A \ar@{^{(}->}[r]\ar[d]_{p_R|_A}& X(\mu)\ar[d]^\pi \\
p_R(A) \ar@{^{(}->}[r]^<<<<<<{\psi|_{p_R(A)}}& X(\mu)\hq G.
}
  \end{xymatrix}
\]
Since $X(\mu)\hq G$ is projective algebraic, $\pi(A) = \psi (p_R(A))$
is Zariski-open in $X(\mu)\hq G$. It follows that $A = p_R^{-1}(p_R^{ }(A))$
is Zariski-open in $U_R$ and contained in $X(\mu)$.
\end{proof}
Let $G$ be a connected algebraic group and let $X$ be an algebraic
$G$-variety. A \emph{Sumihiro neighbourhood} of a point $x \in X$ is a
$G$-invariant, Zariski-open, quasi-projective neighbourhood of $x$ in
$X$ that can be $G$-equivariantly embedded as a locally Zariski-closed subset
of the projective space associated to some rational $G$-representation. In a normal irreducible algebraic $G$-variety, every point has a Sumihiro neighbourhood by \cite{completion}.
\begin{lemma}\label{Zariskiopen}
Let $G$ be a connected complex reductive Lie group and let $X$ be an irreducible algebraic
Hamiltonian $G$-variety with momentum map $\mu: X \to
\mathfrak{k}^*$. Assume that $X(\mu)\hq G$ is projective algebraic and that every point $x
\in \mu^{-1}(0)$ has a Sumihiro neighbourhood. Then, $X(\mu)$ is either empty or
contains a non-empty Zariski-open $G$-invariant subset of $X$.
\end{lemma}
\vspace{-3mm} \begin{proof}Assume that $X(\mu)$ is non-empty.
There exists an irreducible $G$-invariant subvariety $Y$ of $X$ with
$\mu^{-1}(0) \subset Y$ and $Y_{\mathrm{reg}}\cap \mu^{-1}(0) \neq \emptyset$, cf. Lemma \ref{reductiontogeneric}. It follows that
$X(\mu)\hq G = Y(\mu)\hq G$. By Proposition
\ref{prop1} there exists a non-empty Zariski-open $G$-invariant subset $A$ of $Y$ which is contained in
$X(\mu)$ and such that $\pi(A)$ is an open subset of $X(\mu)\hq G$. Furthermore, $A$ consists of
orbits which are closed in $X(\mu)$.

Let $W$ be a Sumihiro neighbourhood of a point $x \in A$ and let $\psi: W \to  \P(V)$ be a $G$-equivariant embedding of $W$ into the projective
space associated to a rational $G$-representation $V$. Let $Z$ be the closure of $\psi(W)$ in $\P(V)$.
Given a subset of $Z$ of the form $U_0$ as in Theorem \ref{Chow}, it follows from Lemma
\ref{saturationconstructible} that $\mathcal{S}^Z_G(\psi(A \cap W))\cap U_0$ is constructible in
$U_0$. Therefore,
$\mathcal{S}_G^X(A \cap W)\cap W$ contains a non-empty Zariski-open subset $\widetilde U$ of its closure. By
construction, $\pi(A \cap W)$ is open in $X(\mu)\hq
G$, and hence, $\pi^{-1}(\pi(A \cap W)) \cap W$ is an open subset of $X$ that is contained in
$\mathcal{S}_G(A \cap W)\cap W$. It follows that $\widetilde{U}$ is  Zariski-open in
$X$, non-empty, and contained in $X(\mu)$.
\end{proof}
\begin{prop}\label{openproposition}
Let $G$ be a connected complex reductive Lie group and let $X$ be an algebraic Hamiltonian
$G$-variety with momentum map $\mu: X \to \mathfrak{k}^*$. Assume that $X(\mu)\hq G$ is projective algebraic and that every point in
$\mu^{-1}(0)$ has a Sumihiro neighbourhood. Then,
$X(\mu)$ is Zariski-open in $X$.
\end{prop}
\vspace{-3mm} \begin{proof}
We may assume that $X(\mu)$ is non-empty. Let $X = \bigcup_j X_j$ be the decomposition of $X$ into irreducible components. Note that
$X_j(\mu|_{X_j}) = X_j \cap X(\mu)$ holds for every $j$ and that $X_j(\mu|_{X_j})\hq G$ is projective
algebraic. There exists a $j_0$ such that $X_{j_0}\bigl(\mu|_{X_{j_0}}\bigr)$ is non-empty. Let $V_{0}$ be the Zariski-open subset of $X_{j_0} \cap X(\mu)$ whose existence is guaranteed by Lemma \ref{Zariskiopen}. Consider  $\widetilde {X} \definiere X \setminus \bigl(V_{0} \setminus \bigcup_{j\neq j_0} X_j \bigr)$. Either $\widetilde X = X \setminus X(\mu)$ or $\widetilde X (\mu|_{\widetilde X}) \neq \emptyset$. In the first case we are done, since $\widetilde X$ is algebraic in $X$. In the second case, we note that $X\setminus X(\mu) \subset \widetilde X \setminus \widetilde X (\mu|_{\widetilde X})$, that $\widetilde X (\mu|_{\widetilde X}) \hq G = \pi (\widetilde X (\mu|_{\widetilde X}))$ is projective algebraic and that the existence of Sumihiro neighbourhoods is inherited by $\widetilde X$. We conclude by Noetherian induction.
\end{proof}

\vspace{-3.5mm} \begin{proof}[Proof of Theorem \ref{unstablepointsarealgebraic}]
Recall that the connected component $G_0$ of the identity of $G$ is the
complexification of the connected component $K_0$ of the identity in $K$. The
set $\mathcal{S}_{G_0}(\mu_0^{-1}(0))$ of semistable points with respect to $G_0$
and the restricted momentum map for the $K_0$-action, $\mu_0 = \mu: X \to \mathfrak{k}= \Lie
(K_0)^*$, coincides with
$X(\mu)$, as we have seen in the proof of Theorem \ref{1}. The finite group $\Gamma = G/G_0$ acts on $X(\mu)\hq G_0$ with quotient $X(\mu)\hq G$, which is projective algebraic by Theorem~\ref{1}. Hence, $X(\mu)\hq G_0$ admits a finite map to a projective algebraic variety and is therefore itself projective-algebraic. The variety $X$ has $1$-rational singularities by assumption; in
particular, it is normal. Consequently, every
point in $\mu^{-1}(0)$ has a Sumihiro neighbourhood for the action of $G_0$. Hence,
Proposition \ref{openproposition} applies and $X(\mu)$ is Zariski-open in $X$.
\end{proof}
\begin{rem}
Example~\ref{noncompactcounterexample} shows that the claim of Theorem~\ref{unstablepointsarealgebraic} is no longer true if $\mu^{-1}(0)$ and $X(\mu)\hq G$ are non-compact.
\end{rem}
\section{Non-zero levels of the momentum map: the shifting trick}\label{section:shifting}
In this section we consider quotients of arbitrary level sets of the momentum map and extend Theorem~\ref{1} and Theorem~\ref{unstablepointsarealgebraic} to this situation with the help of the so-called "shifting trick".

Let $X$ be a complex Hamiltonian $G$-space with momentum map $\mu: X \to \mathfrak{k}^*$ and let $\alpha \in \mu(X) \subset \mathfrak{k}^*$. Then, the level set $\mu^{-1}(\alpha)$ is invariant under the isotropy group $K_\alpha$ of $\alpha$ in the coadjoint representation $\Ad^*: K \to \mathrm{GL}_\R(\mathfrak{k}^*)$. We will recall the classical construction that endows the topological space $\mu^{-1}(\alpha) / K_\alpha$ with a complex structure. The coadjoint orbit $Y = \Ad ^*(K) (\alpha)$ through $\alpha$ carries a $K$-invariant complex structure such that the negative of the Kostant-Kirillov form $\omega_Y$ on $Y$ is K\"ahler. In fact, $Y$ is biholomorphic to $G / P$ for some parabolic subgroup $P$ of $G$; in particular, $Y$ is a projective algebraic manifold. A momentum map for the natural action of $K$ on $Y$ with respect to $-\omega_Y$ is given by $-\imath _Y: Y \hookrightarrow \mathfrak{k}^*, \xi \mapsto - \xi$.

Let $Z \definiere X \times Y$ with the product complex structure and K\"ahler structure $\omega_X - \omega_Y$. The momentum map for the diagonal $K$-action on $Z$ is then given by $\mu_Z (x, \beta) = \mu(x) - \beta$. It follows that $\mu_Z^{-1}(0) = K \acts \left(\mu^{-1}(\alpha) \times \{\alpha\} \right)$. Let $Z(\mu_Z)$ be the corresponding set of semistable points and
\[X( \mu, \alpha) \definiere \{x \in X \mid \overline{G\acts x} \cap \mu^{-1}(\alpha) \neq \emptyset\}.\]
This is an open subset in $X$, cf.\ \cite{Potentialsandconvexity}, Section 6.
By construction, $\mu^{-1}(\alpha) / K_\alpha$ and $\mu^{-1}_Z(0) /K$ are homeomorphic. Consequently, $\mu^{-1} (\alpha)/K_\alpha$ is a K\"ahlerian complex space by Theorem~\ref{propertiesmomentumquotients}.
\begin{thm}\label{shifting}
Let $X$ be an algebraic Hamiltonian $G$-variety with momentum map $\mu: X \to \mathfrak{k}^*$. Assume that $X$ has only $\mathit{1}$-rational singularities. Let $\alpha \in \mu(X)$ such that $\mu^{-1}(\alpha)$ is compact. Then, the complex space $\mu^{-1}(\alpha) / K_\alpha$ is projective algebraic. Furthermore, the set $X( \mu, \alpha)$ of semistable points with respect to $\mu^{-1}(\alpha)$ is algebraically Zariski-open in $X$.
\end{thm}
\vspace*{-3mm}
\begin{proof}
Since the coadjoint orbit $Y$ is smooth, given a resolution $f: \widetilde X \to X$ of $X$, the product map $f \times \id _ Y: \widetilde X \times Y \to X \times Y$ is a resolution of $Z$. It then follows from the K\"unneth formula for coherent algebraic sheaves, see \cite{Kuennethformula}, that $Z$ has only $1$-rational singularities.
Since the complex structure on $\mu^{-1}(\alpha) /K_\alpha$ is given via the homeomorphism to $\mu^{-1}_Z(0)/K \simeq Z(\mu_Z) \hq G$, the first claim follows from the Projectivity Theorem, Theorem ~\ref{1}. For the second claim, we note that the projection $p_X: Z \to X$ maps $Z(\mu_Z) $ onto $X(\mu, \alpha)$. Since $p_X$ is regular and open, it therefore follows from Theorem~\ref{unstablepointsarealgebraic} that $X(\mu, \alpha)$ is Zariski-open in $X$.
\end{proof}
\begin{rem}
Note that Theorem~\ref{shifting} applies in particular to the quotient and the set of semistable points associated to any fibre $\mu^{-1}(\alpha)$ of a proper momentum map $\mu$.
\end{rem}

\section{Embedding $G$-spaces into $G$-vector bundles over
$X\hq G$}\label{sectionembedding}
Let $X$ be a holomorphic $G$-space for a complex reductive Lie group $G$. Assume that the analytic Hilbert quotient $\pi: X \to X\hq G$ exists and that $X\hq G$ is the complex space associated to a projective algebraic variety. In the subsequent sections, we study the implications of this assumption for the equivariant geometry of $X$ and give a proof of Theorem~\ref{2}.
\subsection{Coherent $G$-sheaves and analytic Hilbert quotients}\label{section:coherent}
For a coherent analytic $G$-sheaf $\mathscr{F}$ over $X$ (see e.g.\ \cite{HeinznerCoherent} for the definition) we denote by $(\pi_*\mathscr{F})^G$ the sheaf of invariants on $X \hq G$, i.e., for an open set $Q \subset X\hq G$, we have
\[(\pi_*\mathscr{F})^G (Q) \definiere \mathscr{F}(\pi^{-1} (Q))^G= \{\sigma \in \mathscr{F}(\pi^{-1}(Q)) \mid \sigma(g\acts x) = g\acts \sigma (x)\}.\]
 It is a fundamental result of Roberts \cite{Roberts} and Hausen-Heinzner \cite{HeinznerCoherent} that $(\pi_*\mathscr{F})^G$ is a coherent analytic sheaf on $X\hq G$. An analogous result also holds for the $G$-invariant push-forward of coherent algebraic $G$-sheaves to algebraic Hilbert quotients. Basic examples of coherent analytic $G$-sheaves are provided by sheaves of sections of $G$-vector bundles on a holomorphic $G$-space $X$.
\subsection{The basic construction}
We construct sheaves of equivariant maps that will later on be used to embed a given holomorphic $G$-space into a $G$-vector bundle over $X\hq G$.

Let $L$ be holomorphic line bundle on the quotient $X\hq G$. By GAGA \cite{GAGA}, $L$ is an algebraic line bundle, i.e., there
exists a covering $\{U_\alpha\}_\alpha$ of $X\hq G$ by Zariski-open subsets $U_\alpha$
of $X\hq G$ such that $L|_{U_\alpha}$ is algebraically trivial, and a
collection $\{g_{\alpha \beta}\}_{\alpha, \beta}$ of regular transition functions on
the intersections $U_{\alpha\beta} \definiere U_\alpha \cap
U_\beta$. We will denote the corresponding locally free sheaf by $\mathscr{L}$. Let $V$ be a $G$-module.
Consider the vector bundle $\mathcal{V}$ of rank equal to $\dim V$ that is trivial on $U_\alpha$ and has transition functions $g_{\alpha\beta}\cdot \Id_V: U_{\alpha\beta} \rightarrow  GL_\C(V)$. The bundle $\mathcal{V}$ is isomorphic to the vector bundle $V \tensor L$. Since $g_{\alpha\beta}(U_{\alpha\beta})\Id_V$ is contained in the center of $GL_\C(V)$, the group actions
\[ G \times (U_\alpha \times V) \rightarrow U_\alpha \times V, \;
g \acts (z,v) = (z, g\cdot v)\]
glue to give a global algebraic action $G \times \mathcal{V} \rightarrow \mathcal{V}$.

The action of
$G$ trivially lifts to an action of $G$ on the line bundle $\pi^*L$ by bundle
automorphisms; over $\pi^{-1}(U_\alpha)$ the action is
given by $(x, w) \mapsto (g\acts x, w)$. This action and the diagonal
$G$-action on $X\times V$ yield an action of $G$ on the vector bundle $V \otimes \pi^*L$ by
bundle automorphisms making the bundle projection equivariant. More
precisely, the bundle $V\otimes \pi^*L$ trivialises over
$\{\pi^{-1}(U_\alpha)\}$ via an isomorphism $\phi_\alpha: V \tensor \pi^*L \to \pi^{-1}(U_\alpha) \times V$ and on $(V\otimes
\pi^*L)|_{\pi^{-1}(U_\alpha)} \cong \pi^{-1}(U_\alpha) \times V$
the action is given by
\[ G \times \left( \pi^{-1} (U_\alpha) \times V \right) \rightarrow
\pi^{-1} (U_\alpha) \times V, \quad g\acts (x, v) = (g \acts x, g\cdot
v).\]
It follows that the sheaf $V \tensor \pi^*\mathscr L$ of germs of sections of $ V\otimes \pi^*L$ is a
coherent analytic $G$-sheaf on $X$. We are going to use sections of $\pi_*(V\tensor \pi^* \mathscr{L})^G$ to construct the maps that will yield the embedding of our given holomorphic $G$-space $X$.

\begin{lemma}\label{lemma:basicconstruction}
Let $X$ be a complex $G$-space such that the analytic Hilbert quotient $X\hq G$ exists as a projective algebraic complex space. Let $V$ be a $G$-module. Let $L$ be an algebraic line bundle on $X\hq G$ and $\mathscr{L}$ the associated locally free sheaf.
Then, an element $s \in H^0(X\hq G, \pi_*(V\tensor \pi^* \mathscr{L})^G)$ yields
 a $G$-equivariant holomorphic map $\sigma: X \rightarrow \mathcal{V}$
 into the algebraic $G$-vector bundle $\mathcal{V}=  V \tensor L$ over $X\hq G$ such that the following diagram commutes
\[
\begin{xymatrix}{
 X \ar[rd]_{\pi}\ar[rr]^{\sigma}&   & \mathcal{V}\ar[ld] \\
   & X\hq G. &
}
\end{xymatrix}
\]
\end{lemma}
\vspace*{-3mm}
\begin{proof}
Let $s$ be an element of $H^0(X\hq G, \pi_*(V \tensor \pi^* \mathscr{L})^G)$. Then, the restriction $s|_{\pi^{-1} (U_\alpha)}$ yields a $G$-equivariant map $s_\alpha: \pi^{-1} (U_\alpha) \rightarrow V$  Furthermore, the $s_\alpha$'s satisfy the following compatibility condition on $\pi ^{-1} (U_{\alpha\beta})$: $ s_{\alpha}(x) = g_{\alpha \beta} (\pi(x)) \cdot s_\beta (x)$. One checks that the collection of holomorphic maps defined by $\sigma_\alpha : \pi ^{-1} (U_\alpha) \rightarrow U_\alpha \times V, x \mapsto (\pi(x), s_\alpha(x))$ define a $G$-equivariant holomorphic map $\sigma: X \to \mathcal{V}$.
\end{proof}
Since the bundle projection $\mathcal{V} \to X\hq G$ is $G$-invariant and affine, Corollary $5$ of \cite{BBExistenceOfQuotients} implies that the algebraic Hilbert quotient $\mathcal{V}\hq G$ exists. In fact, it is a Zariski-locally trivial fibre bundle over $X \hq G$ with fibre $V\hq G$.
\subsection{Extending maps from fibres to $X$}
First we note the following immediate consequence of the classical projection formula:
\begin{lemma}[Equivariant Projection Formula]\label{projection}
Let $X$ be a holomorphic $G$-space such that the analytic Hilbert quotient $\pi: X\rightarrow X\hq G$ exists. Let $V$ be a $G$-module, let $\mathscr{F}$ be the coherent analytic $G$-sheaf of germs of maps from $X$ to $V$, and let $\mathscr{L}$ be the locally free sheaf associated to a holomorphic line bundle $L$ on $X\hq G$. Then, we have $(\pi_*\mathscr{F})^G \tensor \mathscr{L} \cong (\pi_*(V \tensor \pi^*\mathscr{L}))^G$.
\end{lemma}
From now on, we fix an ample line bundle $L$ on $X\hq G$. Serre
vanishing allows us  to extend holomorphic maps from the fibres of the
quotient map $\pi$ to
the whole space $X$ after twisting with an appropriate power of
$L$. More precisely, we have
\begin{prop}\label{extendingmaps}
 Let $A= \{ q_1, \dots ,q_N\}$ be a finite set of points in $X\hq G$. Let $F \definiere \pi^{-1} (A)$ be the corresponding collection of fibres of $\pi$. Assume that there exists an equivariant holomorphic map $\varphi$ of a $\pi$-saturated open neighbourhood $U$ of $F$ into a $G$-module $V$ that is an immersion at every point in $F$. Then, there exists $m \in \N$ and
an equivariant holomorphic map $\Phi: X \rightarrow \mathcal{V}(m)$
into the algebraic $G$-vector bundle $\mathcal{V}(m) \definiere  V
\tensor L^{\tensor m}$. The map $\Phi$ coincides with $\varphi$ on $F$ and is an immersion at every point in $F$.
\end{prop}
\vspace*{-3mm}
\begin{proof}
Let $\mathscr{F}$ be the sheaf of germs of holomorphic maps from $X$
into $V$. By the discussion in Section~\ref{section:coherent}, $\mathscr{G} \definiere \pi_*(\mathscr{F})^G$ is a coherent analytic
sheaf of $\mathscr{H}_{X/\negthickspace/G}$-modules on $X\hq G$. Let $\mathfrak{m}_A$ be the ideal sheaf of the
analytic set $A = \{q_1, \dots, q_N\}$ in  $\mathscr{H}_{X/\negthickspace/G}$. Since $\mathscr{L}$ is locally free, the
exact sequence $0 \rightarrow \mathfrak{m}_A^2\mathscr{G} \rightarrow \mathscr{G} \rightarrow \mathscr{G}/\mathfrak{m}_A^2\mathscr{G}\rightarrow 0$
for every $m \in \N$ induces an exact sequence
\begin{equation}\label{exact}
0 \rightarrow \mathfrak{m}_A^2\mathscr{G}\tensor \mathscr{L}^{\tensor m} \rightarrow \mathscr{G}\tensor \mathscr{L}^{\tensor m} \rightarrow \mathscr{G}/\mathfrak{m}_A^2\mathscr{G}\tensor \mathscr{L}^{\tensor m}\rightarrow 0.
\end{equation}
We note that for all $p \in X\hq G$, we have $\left( \mathscr{G}/\mathfrak{m}_A^2\mathscr{G}\tensor \mathscr{L}^{\tensor m}\right)_p \cong \left(\mathscr{G}/\mathfrak{m}_A^2\mathscr{G}\right)_p$. Hence, we get $\mathscr{G}/\mathfrak{m}_A^2\mathscr{G}\tensor \mathscr{L}^{\tensor m} \cong \mathscr{G}/\mathfrak{m}_A^2\mathscr{G}$. Furthermore, by the equivariant projection formula (Lemma \ref{projection}), $\mathscr{G} \tensor \mathscr{L}^{\tensor m}$ is isomorphic to $\pi_*(V \tensor \pi^* \mathscr{L}^{\otimes m})^G$. Hence, from the exact sequence \eqref{exact} we obtain
\begin{equation}\label{exact3}
 0 \rightarrow \mathfrak{m}_A^2\mathscr{G}\tensor \mathscr{L}^{\tensor m} \rightarrow \pi_*(V \tensor \pi^* \mathscr{L} ^{\tensor m})^G \rightarrow \mathscr{G}/\mathfrak{m}_A\mathscr{G} \rightarrow 0.
\end{equation}
By Serre vanishing, there exists an $m_0 \in \N$ such that $H^1(X\hq G, \mathfrak{m}_A^2\mathscr{G}\tensor \mathscr{L}^{\tensor m}) = 0$ for all $m \geq m_0$. Then, for $m\geq m_0$, the exact sequence \eqref{exact3} induces an exact sequence
\begin{equation}\label{exact2}
0 \rightarrow H^0(X\hq G, \mathfrak{m}_A^2\mathscr{G}\tensor \mathscr{L}^{\tensor m}) \rightarrow H^0(X\hq G, \pi_*(V \tensor \pi^* \mathscr{L} ^{\tensor m})^G)\overset{r}{\rightarrow} H^0( X\hq G, \mathscr{G}/\mathfrak{m}_A^2\mathscr{G})\rightarrow 0,
\end{equation}
where the map $r$ is given by restriction. The map $\varphi: U \rightarrow V$ defines $\varphi_q \in H^0(X\hq G,
\mathscr{G}/\mathfrak{m}_A^2\mathscr{G})$, and it follows from the exact sequence \eqref{exact2} that there exists a $\phi \in H^0(X\hq G, \pi_*(V \tensor \pi^*\mathscr{L}^{\tensor m})^G)$ such that $r(\phi) = \varphi_q$. Shrinking $U$ if necessary, we can assume that $L|_U$ is trivial and
hence that $\phi|_{\pi(U)}$ is given by a G-equivariant map from  $U$ to $V$. The difference $\phi|_{\pi(U)} - \varphi$ vanishes along $F$ up to second order. Lemma \ref{lemma:basicconstruction} yields a $G$-equi\-va\-riant holomorphic map $\Phi: X \rightarrow V \otimes L^{\tensor m}$. We note that by construction, $\Phi|_U: U \rightarrow \pi(U) \times V$ coincides with $\pi \times \varphi : U \rightarrow \pi(U) \times V$ up to second order along $F$.
\end{proof}
\subsection{The Embedding Theorem}\label{section:embeddingtheorem}
In this section, we will prove the first main result of this chapter: a holomorphic $G$-space with
projective algebraic quotient $X\hq G$ admits a $G$-equivariant embedding into a $G$-vector bundle
$\mathcal{V}$ over $X\hq G$. This is the analogue of Kodaira's embedding theorem in our context.

Let $G$ be a complex reductive Lie group and let $X$ be a holomorphic $G$-space. Let $G_x$, $G_y$ be two
isotropy subgroups for the action of $G$ on $X$. We define a preorder on the set of $G$-orbits in $X$
as follows: we say $G\acts x \leq G \acts y$ if and only if there exists an element $g \in G$ such that
$gG_yg^{-1} < G_x$. This preorder induces
an equivalence relation on the set of $G$-orbits in $X$ given as follows: $G\acts x \sim G \acts y$
if and only if there exists an element $g \in G$ such that $gG_yg^{-1} = G_x$. The equivalence class of an
orbit $G\acts x$ will be denoted by $\Type(G\acts x)$ or $\Type(G/G_x)$. We call $\Type(G/G_x)$ the
\emph{orbit type} of $x$. The preorder $\leq$ defined above for $G$-orbits induces a preorder on the
set of orbit types.

Two representations $\rho_i: H_i \rightarrow GL(V_i);\,i=1,2$ of two closed complex subgroups $H_1,
H_2$ of $G$ will be called \emph{$G$-isomorphic}, if there exists $g\in G$ and a linear
isomorphism $L: V_1 \rightarrow V_2$ such that
\[i)\; H_2 = g H_1 g^{-1}, \quad ii)\; L(\rho_1(h)\cdot v)= \rho_2(ghg^{-1})\cdot L(v).\]This is
the case if and only if the associated $G$-vector bundles $G \times_{H_i}V_i$ are $G$-isomorphic.
The equi\-va\-lence class of a representation $\rho:H \rightarrow GL(V)$ under this equivalence relation as well as the associated $G$-isomorphism class of $G$-vector bundles represented by $G \times_H V$
is called the \emph{slice type} of $\rho$ and denoted by $\Type_G(V,H)$.

A holomorphic $G$-space $X$ that admits an analytic Hilbert quotient will be called \emph{of finite
orbit type}, if the set $\{\Type(G\acts x) \mid x \in X,\; G\acts x \text{ is closed in }X \}$ is
finite. If the set of slice types $\{\Type_G(T_xX, G_x)\mid x \in X,\; G\acts x
\text{ is closed in } X \}$ is finite, we say that $X$ is \emph{of finite slice type}.

It follows from the holomorphic Slice Theorem that a holomorphic $G$-space with compact analytic Hilbert quotient is of finite orbit type and of finite slice type, see \cite{HeinznerEinbettungen}.

The next result is the crucial technical part of the Embedding Theorem:
\begin{lemma}\label{technical}
Let $X$ be a holomorphic $G$-space with analytic Hilbert quotient $\pi: X \to X\hq G$. Assume that $X\hq
G$ is projective algebraic. Then, for every $d$-dimensional analytic subset $A$ of $X\hq G$
the following holds:
\vspace{-2.5mm}
\begin{enumerate}
 \item There exists an analytic subset $\widetilde{A} \subset A$, $\dim \widetilde{A} < d$, an algebraic $G$-vector
bundle $\mathcal{V}$ over $X\hq G$, and an equivariant holomorphic map $\Phi: X \rightarrow
\mathcal{V}$ that is an immersion along $\pi^{-1} (A \setminus \widetilde{A})$.
 \item There exists an analytic subset $\check{A} \subset A$, $\dim \check{A} < d$, an algebraic $G$-vector
bundle $\mathcal{W}$ over $X\hq G$, and an equivariant holomorphic map $\Psi: X \rightarrow
\mathcal{W}$ whose restriction to every closed $G$-orbit in $\pi ^{-1} (A \setminus \check{A})$ is a
proper embedding.
\end{enumerate}
\end{lemma}
\vspace*{-3mm}
\begin{proof}The proof goes along the lines of the proof of the corresponding result in the Stein
case (see \cite{HeinznerEinbettungen}).
Without loss of generality, we can assume that $A$ is pure-dimensional. Let $A_i$, $i \in I$ be the
irreducible components of $A$.

$(1)$ For every $i\in I$, choose a point $p_i \in A_i \setminus
   \bigcup_{i \neq j}A_j$. The set $D = \{p_i \mid i \in I\}$ is discrete,
   hence finite and analytic in $A$. For every $p \in D$, the fibre $\pi^{-1}(p)$
   contains exactly one closed orbit $G\acts x$ and we set the
   \emph{slice type of $p$} to be the slice type of $\rho_x: G_x \rightarrow GL(T_xX)$. The set $D$
is a finite union of pairwise disjoint analytic sets $D_\alpha$ such that each $D_\alpha$ contains
points of only one particular slice type. Hence, $A$ is the finite union $A = \bigcup_\alpha
A(\alpha)$ where $A(\alpha)$ is the union of those irreducible components of $A$ that contain a
point of $D_\alpha$.

For fixed $\alpha$,  choose a representative $G \times_{H_\alpha} T_\alpha$ of the slice type
of $D_\alpha$. Since $G\times_{H_\alpha}T_\alpha$ is an affine $G$-variety, there exists a proper equivariant embedding of $G\times_{H_\alpha}T_\alpha$ into a
complex $G$-module $V_\alpha$. Since $D_\alpha$ is finite, there exists a covering $\{U_{\alpha s}\}$ of $D_\alpha$ by open subsets $U_{\alpha s}$ of $X$, such that $U_{\alpha s}\cap U_{\alpha t} =
\emptyset$ for $s\neq t$. If we choose the $U_{\alpha s}$ small enough, by the holomorphic Slice Theorem,
there exists an equivariant holomorphic embedding of $\pi^{-1} (U_{\alpha s})$ into a
saturated open subset of the $G$-module $V_\alpha$. By Proposition \ref{extendingmaps}, we get an
equivariant holomorphic map $\Phi_\alpha: X \rightarrow \mathcal{V}_\alpha$ that is an immersion
at every point in $\pi^{-1} (D_\alpha)$. The set $R_\alpha \definiere \{x \in X \mid \Phi_\alpha \text{ is not an immersion in }x\}$ is a
$G$-invariant analytic subset of $X$. Furthermore, we consider the analytic subset $\widetilde{A} \definiere \bigcup_\alpha \left(\pi(R_\alpha) \cap A(\alpha)  \right)$ of $A$. Since the map $\Phi_\alpha$ is an immersion at every point in $\pi^{-1}(D_\alpha)$,
we conclude that $\dim \widetilde{A} < d = \dim A$. The product map
$\Phi \definiere \bigoplus_\alpha \Phi_\alpha: X \longrightarrow  \bigoplus_\alpha
\mathcal{V}_\alpha$
into the $G$-product $\bigoplus_\alpha
\mathcal{V}_\alpha \rdefiniere \mathcal{V}$ of the $\mathcal{V}_\alpha$ is an immersion at every point in $\pi^{-1} (A \setminus
\widetilde{A})$.

$(2)$ For $p \in X\hq G$, let $\Type (p)$ be the orbit type of the uniquely determined closed orbit in
$\pi^{-1} (p)$ and for any subset $Y \subset X\hq G$, let $\Type Y \definiere \{\Type(p) \mid p \in
Y\}$. Note that $\Type(X\hq G)$ is finite, since $X\hq G$ is compact. For any
$Y\subset X\hq G$, we have $\Type Y \subset \Type (X\hq G)$, which implies that $\Type Y$ is
finite.

Inductively, we define the following finite collection of analytic subsets $E_\alpha \subset A$ with
$\bigcup_\alpha E_\alpha = A$:
set $I_0 \definiere I$ and for $I_\alpha \subset I$, set $A(\alpha) \definiere \bigcup_{i \in
I_\alpha} A_i$. Furthermore, we define
\begin{align*}
 \mathcal{H}_\alpha &\definiere \left\{\Type (p) \mid \Type (p) \text{ is maximal in } \Type
(A(\alpha))\, (\text{w.r.t.\ } \leq)\right\} \\
 J_\alpha           &\definiere \{j \in I_\alpha \mid \Type A_j \cap \mathcal{H}_\alpha \neq
\emptyset\}\\
 E_\alpha           &\definiere \bigcup_{j \in J_\alpha} A_j
\end{align*}
and finally, we set $I_{\alpha +1} \definiere I_\alpha \setminus J_\alpha$.

Since $\Type A$ is finite, this procedure stops after finitely many steps. In $E_\alpha$ we choose a
finite set $D_\alpha$ with the following two properties:
\[i)\; \Type (p) \in \mathcal{H}_\alpha \text{ for all }p \in D_\alpha \quad \quad
 ii)\; D_\alpha \cap A_j \neq \emptyset \text{ for all }j \in J_\alpha.
\]
We claim that there exists a $G$-module $W_\alpha$ such that every orbit type $\Type (p) \in
\mathcal{H}_\alpha$ is represented by a closed orbit $G\cdot w_p$ in
$W_\alpha$. Indeed, any closed $G$-orbit $G \acts x$ in $X$ naturally carries the structure
of an affine algebraic $G$-variety, since $G_x$ is reductive. Hence, it can be $G$-equivariantly embedded as a closed orbit in a
$G$-module $V$. Since $\mathcal{H}_\alpha \subset \Type A$ is finite, we can find a $G$-module $W_\alpha$ with the desired properties.

For every $p \in D_\alpha$, the closed orbit $T(p)$ in $\pi^{-1} (p)$
is equivariantly biholomorphic to $G\cdot w_p \subset
W_\alpha$. Hence, we have a $G$-equivariant holomorphic map from the
disjoint union of the $T(p)$, $p\in D_\alpha$ into $W_\alpha$. This
map extends to a $G$-equivariant holomorphic map from a saturated
Stein neighbourhood of $\pi^{-1} (D_\alpha)$ into $W_\alpha$ (see
\cite{HeinznerEinbettungen}, 1.\ Prop.\ 1). Hence, Proposition
\ref{extendingmaps}
yields a $G$-equivariant holomorphic map $\Psi_\alpha: X \rightarrow
\mathcal{W}_\alpha$ into an algebraic $G$-vector bundle
$\mathcal{W}_\alpha$ over $X\hq G$ with the pro\-perty that
$\Psi_\alpha(T(p)) = G\cdot w_p \subset (\mathcal{W}_\alpha)_p \cong W_\alpha$ for all $p\in
D_\alpha$.

Note that if there exists a $G$-equivariant map $\varphi: G\acts x
\rightarrow G\acts y$ between two $G$-orbits in $G$-spaces, then $\Type(G\acts y) \leq \Type( G\acts x)$. Therefore, for every closed orbit $G\acts y \subset X$ with $q \definiere \pi(y) \in E_\alpha$, we have
\begin{equation}\label{A}
\Type(G\cdot \Psi_\alpha(y)) \leq \Type (G\acts y) = \Type (q) \leq \Type(p) = \Type (G\cdot w_p) \tag*{($\ast$)}
\end{equation}
for some $p \in \mathcal{H}_\alpha$.

By Corollary $5$ of \cite{BBExistenceOfQuotients}, the algebraic Hilbert quotient $\pi_\alpha: \mathcal{W}_\alpha \to \mathcal{W}_\alpha\hq G$ exists.

\textbf{Claim:} In $\mathcal{W}_\alpha$ there exists a $\pi_{\mathcal{W}_\alpha}$-saturated analytic
subset $\mathcal{W}_\alpha '$ with the following two properties:
\vspace{-2.5mm}
\begin{enumerate}
 \item[(a)] $G\acts \Psi_\alpha(p) \subset \mathcal{W}_\alpha \setminus \mathcal{W}_\alpha
'\quad \text{ for all }p \in X\hq G$ with $\Type(p) \in
 \mathcal{H}_\alpha$
 \item[(b)] If $w \in \mathcal{W}_\alpha \setminus \mathcal{W}_\alpha '$ and $\Type(G\acts w) \leq
\Type(G\cdot w_p)$ for some $p \in D_\alpha$, then $\Type (G\acts w) = \Type (G\cdot w_p)$ and
$G\acts w$ is closed in $\mathcal{W}_\alpha$.
\end{enumerate}
Assuming the claim, we finish the proof of the lemma as follows: set $\check{A}^\alpha \definiere
\pi(\Psi_\alpha^{-1} (\mathcal{W}_\alpha '))$ and, in the setup of \ref{A}, let $G \acts y$ be a closed orbit such that in addition $\pi(y) \in E_\alpha \setminus \check{A}^\alpha$ holds. Then, \ref{A} shows that
$\Type (G\acts \Psi_\alpha(y)))\leq \Type (G\acts w_p)$ for some $p \in D_\alpha$ and $\Psi_\alpha
(y) \in \mathcal{W}_\alpha \setminus \mathcal{W}_\alpha '$. Hence, by part (b) of the claim, it
follows that $\Type (G\acts \Psi_\alpha(y)) = \Type (G\acts w_p)$ and $G\acts \Psi_\alpha(y)$ is
closed in $\mathcal{W}_\alpha$. This implies that $\Psi_\alpha|_{G\acts y }$ is a proper embedding.
From part (a) it follows that for every $j \in J_\alpha$, the set $A_j$ contains a point of $E_\alpha
\setminus \check{A}^\alpha$. Hence, $\dim \check{A}^\alpha < d$.

We define $\check{A} \definiere \bigcup_\alpha \check{A}^\alpha$ and consider the map
$\Psi \definiere \bigoplus_\alpha \Psi_\alpha: X \longrightarrow
\bigoplus_\alpha \mathcal{W}_\alpha \definiere \mathcal{W}$. We claim
that $\Psi$ has the desired properties. Indeed, let $G\acts x$ be a
closed orbit in the set $\pi^{-1} (A \setminus \check{A})$. Then, $\pi(x) \in
E_\alpha \setminus \check{A}^\alpha$ for some $\alpha$ and the restriction
of $\Psi_\alpha: X \rightarrow \mathcal{W}_\alpha$ to $G\acts x$ is a
proper embedding. We have $\Psi_\alpha(G\acts x) \subset
(\mathcal{W}_\alpha)_{\pi(x)}\cong W_\alpha$. It follows from \cite{HeinznerGIT}, Lemma 6.2, that the restriction of $\Psi$ to $G \acts x$ is a proper embedding.

It remains to prove the claim. We omit the index $\alpha$. Since the quotient
$\pi_{\mathcal{W}}: \mathcal{W} \rightarrow \mathcal{W}\hq G$ is locally the
quotient of an affine $G$-variety, it follows from Luna's Slice Theorem \cite{LunaSlice} that
the set
\[(\mathcal{W}\hq G)' \definiere \{q \in \mathcal{W}\hq G  \mid \Type(q) < \Type(z) \text{ for some
} z \in D\}\]
is an algebraic subset of $\mathcal{W}\hq G$. Let $\mathcal{W}'\definiere \pi^{-1}((\mathcal{W}\hq G)')$. Since $\Type (G\cdot w_p) = \Type (p)$ for all $p$ with $\Type(p) \in
\mathcal{H}$, the set $\mathcal{W}'$
has property (a). It remains to show property (b): let $w \in
\mathcal{W}\setminus \mathcal{W}'$ such that $\Type(G\acts w) \leq
\Type(G\cdot w_p)$ for some $p$ with $\Type(p) \in \mathcal{H}$. Suppose that $G\acts
w$ is not closed. Then there exists a unique closed $G$-orbit $G\acts
\widetilde{w} \subset \overline{G\acts w}$. Since the $G$-action on $\mathcal{W}$ is algebraic, we have
$\Type(G\acts \widetilde{w}) <
\Type(G\acts w)\leq \Type(G\cdot w_p)$. This implies $G\acts \widetilde{w}
\subset \mathcal{W}'$. However, this is a contradiction to the fact that $\mathcal{W}\setminus
\mathcal{W}'$ is $\pi_{\mathcal{W}}$-saturated. Hence, $G\acts w$ is closed. The definition of
$(\mathcal{W}\hq G)'$ now implies that $\Type(G\acts w) = \Type(G\cdot w_p)$. This shows the claim
and completes the proof of the lemma.
\end{proof}
In addition to Lemma \ref{technical}, the holomorphic version of Luna's fundamental lemma (see
\cite{HeinznerEinbettungen}, 2.\ Prop.\ 2,  and \cite{Snow}, Prop.\ 5.1) is an essential ingredient of the proof of the Embedding Theorem. For convenience of the reader we recall the statement.
\begin{lemma}[\cite{HeinznerEinbettungen}]\label{fundamental}
Let $X,Y$ be holomorphic $G$-spaces such that the analytic Hilbert quotients $\pi_X: X
\rightarrow X\hq G$ and $\pi_Y: Y \rightarrow Y\hq G$ exist. Let $x \in X$ be a point such that the
$G$-orbit $G\acts x$ through $x$ is closed. Let $\phi: X \rightarrow Y$ be an equivariant
holomorphic map with the following properties:
\vspace{-1mm}
\begin{enumerate}
 \item $\phi$ is an immersion at $x$, and
 \item $\phi|_{G\acts x}: G\acts x \rightarrow Y$ is a proper embedding.
\end{enumerate}
Then, given an open neighbourhood $U$ of $x$, there exists an open $\pi_X$-saturated neighbourhood $T$ of $x$ in
$X$ and a $\pi_Y$-saturated open subset $T'$ of $Y$ such that $\pi_X(T) \subset \pi_X(U)$, $\phi(T)
\subset T'$ and such that $\phi|_T:T \rightarrow T'$
is a proper holomorphic embedding.
\end{lemma}

\begin{thm}[Embedding Theorem]\label{embedding}
Let $X$ be a holomorphic $G$-space with analytic Hilbert quotient $X\hq G$. Assume that
$X\hq G$ is projective algebraic. Then, there exists an equivariant proper holomorphic embedding
$\Phi:~X \rightarrow~\mathcal{V}$ of $X$ into an algebraic $G$-vector bundle $\mathcal{V}$ over $X\hq G$.
\end{thm}
\vspace*{-3mm}
\begin{proof}
Apply Lemma \ref{technical} to $A = X$ to produce an equivariant holomorphic map $\Psi_1: X
\rightarrow \mathcal{W}_1$ that is an immersion outside of an analytic set $\widetilde{A}$ of dimension less
than $\dim X$. Applying the lemma again to $A= \widetilde{A}$, we get a second map $\Psi_2: X \rightarrow
\mathcal{W}_2$. The product map $\Psi_1 \oplus \Psi_2: X \rightarrow \mathcal{W}_1
\oplus \mathcal{W}_2$ is an immersion outside an analytic set of dimension less than $\dim
\widetilde A$. Hence, after finitely many steps we obtain an equivariant holomorphic immersion
$\Psi: X \rightarrow \mathcal{W}$ into an algebraic $G$-vector bundle over $X\hq G$. Furthermore, by a repeated application of the second part of Lemma \ref{technical} we get an equivariant holomorphic map $\widetilde{\Psi}: X \rightarrow \widetilde{\mathcal{W}}$ into an algebraic $G$-vector bundle $\widetilde{\mathcal{W}}$
over $X\hq G$, whose restriction to every closed orbit in $X$ is a proper embedding.

Consider $\mathcal{V}\definiere \mathcal{W} \oplus
\widetilde{\mathcal{W}}$ and let $\Phi \definiere \Psi \oplus \widetilde\Psi: X \rightarrow
\mathcal{V}$ be the product map. Then, $\Phi$ is an immersion and its restriction to every closed
orbit in $X$ is a proper embedding. From Lemma
\ref{fundamental} it follows that the restriction of $\Phi$ to every fibre of $\pi: X\rightarrow
X\hq G$ is a proper embedding. Furthermore, by construction, $\Phi$ separates the fibres of $\pi$.
Hence, $\Phi$ is an injective immersion; it remains to show that it is proper.

Let $\overline{\Phi}: X \hq G \to \mathcal{V}\hq G$ be the holomorphic map induced by the $G$-invariant holomorphic map $\pi_{\mathcal{V}} \circ \Phi: X \to \mathcal{V}\hq G$, where $\pi_{\mathcal{V}}: \mathcal{V} \to \mathcal{V}\hq G$ is the algebraic Hilbert quotient. Since $X\hq G$ is compact, $\overline{\Phi}$ is a proper embedding. Let $(x_n)$ be a sequence in $X$ such
that $\Phi(x_n) \rightarrow v_0 \in \mathcal{V}$. Then, $\pi_{\mathcal{V}}(\Phi(x_n))
\rightarrow \pi_{\mathcal{V}}(v_0)$ in $\mathcal{V}\hq G$. Since $\overline{\Phi}$ is proper, without loss
of generality, $\pi(x_n)$ converges to a point $q_0$ in $X\hq G$ and $\overline{\Phi}(q_0) =
\pi_{\mathcal{V}}(v_0)$. Hence, given an arbitrary open neighbourhood $Q$ of $q_0$ in $X\hq G$,
without loss of generality we can assume that $x_n \in \pi^{-1}(Q)$ for all $n$. Choose $Q$ to be a
neighbourhood of the form $\pi(T)$ as given by Lemma \ref{fundamental}, and
let $T'$ be the corresponding $\pi_{\mathcal{V}}$-saturated neighbourhood of $v_0$ in
$\mathcal{V}$. It
follows that $\Phi(x_n) \in T'$ for all $n\in \N$. Since $\Phi|_T:
T\rightarrow T'$ is a proper embedding, the convergence of $\Phi(x_n)$ to $v_0$ in $T'$ implies
that a subsequence of
$(x_n)$ converges to some $x_0$ in $T \subset X$.
\end{proof}
\section{Algebraicity of spaces with projective algebraic
quotient}\label{sectionalgebraicityofspaces}
\subsection{The Algebraicity Theorem for invariant analytic subsets}
Let $Y$ be an algebraic $G$-variety with algebraic Hilbert quotient $\pi: Y \rightarrow Y\hq G$.
Assume that $Y\hq G$ is a projective algebraic variety. Let $X$ be a $G$-invariant analytic subset
of $Y$ such that $\pi(X)=Y\hq G$. Fix an ample line bundle $L$ on $Y\hq G$. Given a vector bundle $\mathcal{V}$ on a complex space $Z$ let $\mathcal{Z}_\mathcal{V} \subset \mathcal{V}$ denote the zero section.
\begin{lemma}\label{Bundleseparation}
For every $y \in Y \setminus X$, there exist a $G$-module $V$, $m \in \N$, and a $G$-equivariant
holomorphic map $\Phi: Y \rightarrow \mathcal{V}$ into the $G$-vector bundle $\mathcal{V} = V \tensor L^{\tensor m}$ such
that $\Phi(X) \subset \mathcal{Z}_\mathcal{V}$ and $\Phi(y) \notin \mathcal{Z}_\mathcal{V}$.
\end{lemma}
\vspace*{-3mm}
\begin{proof}
Let $y \in Y \setminus X$. Since $\pi(X) = Y\hq G$, the fibre of $\pi$ over $q \definiere \pi(y)$
intersects $X$ non-trivially. Set $A \definiere X \cap \pi^{-1}(q)$. Since $\pi^{-1}(q)$ is Stein, it follows from \cite{HeinznerEinbettungen}, 1.\ Bem.\ 2, that there exists a $G$-module $V$ and a
$G$-equivariant holomorphic map $\varphi: \pi^{-1}(q) \rightarrow V$ such that $\varphi|_A \equiv 0$
and $\varphi(y) \neq 0$. Consider the sheaf $\mathscr{F}$ of germs of holomorphic maps from $Y$ to
$V$ and let $\mathscr{G}\subset \mathscr{F}$ be the subsheaf of those germs that vanish along $X
\subset Y$, i.e., $\mathscr{G} = \mathscr{I}_X\cdot \mathscr{F}$, where $\mathscr{I}_X$ denotes the
ideal sheaf of $X$ in $\mathscr{H}_Y$. Since $X$ is $G$-invariant, $\mathscr{G}$ is
a coherent analytic $G$-sheaf on $Y$. The $G$-invariant push-forward $\widetilde{\mathscr{G}} \definiere
(\pi^h_*\mathscr{G})^G$ is a coherent analytic sheaf on $X\hq G = Y\hq G$. Let $\mathfrak{m}_q$ be
the ideal sheaf in $\mathscr{H}_{Y/\negthickspace/ G}$ of the point $q \in Y\hq G$. Let $m \in \N$ be chosen such that
$H^1(Y\hq G, \mathfrak{m}_q\cdot \widetilde{\mathscr{G}}\tensor \mathscr{L}^{\tensor
  m})=0$ (Serre vanishing). Then, the exact sequence $0 \rightarrow \mathfrak{m}_q \cdot  \widetilde{\mathscr{G}} \rightarrow \widetilde{\mathscr{G}}
\rightarrow \widetilde{\mathscr{G}}/\mathfrak{m}_q\cdot \widetilde{\mathscr{G}} \rightarrow 0$
induces an exact sequence on the level of sections
\[0 \rightarrow H^0(Y\hq G, \mathfrak{m}_q \cdot  \widetilde{\mathscr{G}}\tensor \mathscr{L}^{\tensor m})
\rightarrow H^0(Y\hq G, \widetilde{\mathscr{G}}\tensor \mathscr{L}^{\tensor m}) \rightarrow H^0(Y\hq G,
\widetilde{\mathscr{G}}/\mathfrak{m}_q\cdot \widetilde{\mathscr{G}}) \rightarrow 0. \]The map $\varphi$
defines an element in $H^0(Y\hq G, \widetilde{\mathscr{G}}/\mathfrak{m}_q\cdot \widetilde{\mathscr{G}})$ and
hence, there exists a section of $\widetilde{\mathscr{G}}\otimes L^{\otimes m}$ that coincides with
$\varphi$ at $q$. Lemma \ref{lemma:basicconstruction} now yields a $G$-equivariant holomorphic map $\Phi: Y \rightarrow V \tensor L^{\otimes m}$ with $\Phi(X) \subset \mathcal{Z}_{\mathcal{V}}$ and $\Phi(y) \notin \mathcal{Z}_{\mathcal{V}}$.
\end{proof}
In the setup of the previous lemma, if $\Phi(X) \subset \mathcal{Z}_\mathcal{V}$, we say that $\Phi$ \emph{vanishes on $X$}. Analogously, if $\Phi(y) \in \mathcal{Z}_\mathcal{V}$, we write $\Phi(y) = 0$.

Next, we need to compare coherent algebraic and coherent analytic sheaves on algebraic varieties. For this, we introduce some notation.  If $\mathscr{F}$ is any sheaf on an algebraic variety $Y$ we define a new sheaf $\mathscr{F}'$ on $Y^h$: for an open subset $U \subset Y$ we define $\mathscr{F}' (U) \definiere \underset{\rightarrow}{\lim}\{\mathscr{F}(W) \mid W
\text{ Zariski-open in $Y$ containing $U$}\}$.
If $\mathscr{F}$ is any algebraic sheaf on the algebraic variety $Y$, then we define
a corresponding analytic sheaf $\mathscr{F}^h$  on $Y^h$ by $\mathscr{F}^h \definiere \mathscr{H}_Y
\tensor_{\mathscr{O}'_Y} \mathscr{F}'$.

\begin{lemma}\label{equivariantalgebraic}
The holomorphic map $\Phi: Y \rightarrow \mathcal{V}$ constructed in the
previous lemma is algebraic.
\end{lemma}
\vspace*{-3mm}
\begin{proof}
We have $\mathscr{G} = \mathscr{I}_X\cdot
\mathscr{F} \subset \mathscr{F}$, where $\mathscr{F}$ is the coherent analytic $G$-sheaf of germs of
holomorphic maps from $Y^h$ to $V$ and hence $\widetilde{\mathscr{G}} \subset (\pi^h_*\mathscr{F})^G$
due to the exactness of the $\pi^h_*(\cdot)^G$-functor, cf.\ \cite{HeinznerCoherent}. Let $L$ be the ample line bundle introduced above, and let $\mathscr{L}$ be the locally free sheaf associated to $L$. The map
$\Phi$ is induced by an element $\phi \in H^0 (Y\hq G, (\pi_*\mathscr{F})^G \tensor \mathscr{L}^{\tensor
m})$. Let $\mathscr{A}$ be the coherent algebraic $G$-sheaf of germs of algebraic maps from $Y$ to
the $G$-module $V$. Then, $\mathscr{F}$ is isomorphic  to $\mathscr{A}^h$ as a coherent analytic
$G$-sheaf over $Y^h$. We
have
\[((\pi_*\mathscr{A})^G \tensor_{\mathscr{O}_{Y/\negthickspace/ G}} \mathscr{L}^{\tensor m})^h = ((\pi_*\mathscr{A})^G)^h
\tensor_{\mathscr{H}_{Y/\negthickspace/ G}}\mathscr{L}^{\tensor m} = (\pi_*^h\mathscr{F})^G \tensor_{\mathscr{H}_{Y/\negthickspace/
G}}\mathscr{L}^{\tensor m},\] where for the second equality we used the
GAGA-principle for quotient morphisms, see \cite{NeemanAnalytic}. Since $Y\hq G$ is projective algebraic, the section $\phi \in H^0 ((Y\hq G)^h, (\pi_*^h\mathscr{F})^G
\tensor_{\mathscr{H}_{Y/\negthickspace/ G}} \mathscr{L}^{\tensor m})$ is an element of $H^0(Y\hq G, (\pi_*\mathscr{A})^G
\tensor_{\mathscr{O}_{Y/\negthickspace/ G}} L^{\tensor m})$ and hence algebraic.
\end{proof}
We are now in the position to prove to the main result of this section. It is the analogue of Chow's Theorem in our context.
\begin{thm}[Algebraicity Theorem for invariant analytic subsets]\label{invariantalgebraic}
Let $Y$ be an algebraic $G$-variety with algebraic Hilbert quotient $\pi: Y \rightarrow Y\hq G$.
Assume that $Y\hq G$ is projective algebraic. Let $X \subset Y$ be a $G$-invariant
analytic subset of $Y$ such that $\pi(X) = Y\hq G$. Then, $X$ is an algebraic subvariety of $Y$.
\end{thm}
\vspace*{-3mm}
\begin{proof}Set $\mathcal{A}\definiere \{\Phi: Y \rightarrow V \tensor L^{\tensor m} \mid \Phi
\text{ a } G\text{-equivariant algebraic map}, V \text{ a } G\text{-module}, m\in \N\}$.
It suffices to show that $X = \{y \in Y \mid \Phi(y) = 0 \text{ for every }\Phi \in \mathcal{A} \text{ vanishing on
}X\} \rdefiniere B$, since then $X$ is the intersection of a family of algebraic subvarieties of $Y$,
hence algebraic.
Clearly, we have $X \subset B$. So, let $y \in Y \setminus X$. Then, using Lemmata
\ref{Bundleseparation} and \ref{equivariantalgebraic}, we see that there exists a map $\Phi \in \mathcal{A}$
such that $\Phi$ vanishes on $X$ and $\Phi(y) \neq 0$. Hence, $y \notin B$ and the claim is shown.
\end{proof}
Lemma \ref{Bundleseparation} and Theorem \ref{invariantalgebraic} are no longer true if the quotient is not assumed to be projective:
\begin{ex}
Consider the action of $\C^*$ on $\C^3 = \C \times \C^2$ by multiplication in the
second factor. The analytic Hilbert quotient is given by $\pi: \C \times \C^2 \to \C$, $(z,w) \mapsto z$.
The holomorphic function $F: \C^3 \to \C$, $(u_1, u_2, u_3) \mapsto u_3 - u_2 \cdot
e^{u_1}$ is equivariant with respect to the $\C^*$-action on $\C^3$ and the standard action of
$\C^*$ on $\C$ by multiplication. Its zero set $X = \{F = 0\}$ is a $\C^*$-invariant, non-algebraic,
analytic subset of $\C^3$ with $\pi(X) = \C^3 \hq \C^* = \C$.
\end{ex}

\subsection{The Algebraicity Theorem}\label{algebraicitytheorem}
In this section we prove Theorem \ref{2}. We have separated the statements into two propositions.

\begin{prop}\label{existsalgstructure}
Let $X$ be a holomorphic $G$-space such that the analytic Hilbert quotient $\pi: X \rightarrow X\hq
G$ exists and such that $X\hq G$ is projective algebraic. Then, there exists a quasi-projective algebraic $G$-variety $X_0$
with algebraic Hilbert quotient $\pi_0:X_0 \rightarrow X_0\hq G \cong X\hq G$ such that $X$ is $G$-equivariantly biholomorphic to $X^h_0$.
\end{prop}
\vspace*{-3mm}
\begin{proof}
By the Embedding Theorem, Theorem \ref{embedding}, there exists a proper holomorphic embedding
$\Phi: X \hookrightarrow \mathcal{V}$ into an algebraic $G$-vector bundle over $X\hq G$.
Furthermore, $\mathcal{V}$ admits an algebraic Hilbert quotient $\pi_{\mathcal{V}}: \mathcal{V}
\rightarrow \mathcal{V}\hq G$. Since $X\hq G$ is projective, the induced map $\overline{\Phi}: X\hq G \to \mathcal{V}\hq G$ is algebraic and its image is a compact algebraic subvariety of $\mathcal{V}\hq G$. Set $Y \definiere \pi_{\mathcal{V}}^{-1}
(\overline{\Phi}(X\hq G))\subset \mathcal{V}$. The restriction $\pi_Y\definiere\pi_{\mathcal{V}}|_Y: Y
\rightarrow \pi_{\mathcal{V}}(Y) \cong X\hq G$ is an algebraic Hilbert quotient, and $\Phi:X
\rightarrow Y$ is a proper holomorphic embedding.

Therefore, we can assume that $X$ is a $G$-invariant analytic subset of a
quasi-projective $G$-variety $Y$ that admits an algebraic Hilbert quotient $\pi_Y: Y\rightarrow Y\hq
G$ such that $\pi(X) = Y\hq G$ is projective algebraic. We conclude that $\pi= \pi_Y^h|_X$. The Algebraicity Theorem for invariant
analytic sets, Theorem~\ref{invariantalgebraic}, implies that $X$ is an algebraic subset of $Y$,
hence a quasi-projective algebraic variety.
\end{proof}
By Proposition~\ref{existsalgstructure} we can endow $X$ with the structure of an algebraic $G$-variety with algebraic Hilbert quotient $X\hq G$. Next, we investigate uniqueness properties of this algebraic structure.
\begin{prop}\label{twoquotients}
Let $X$ be a holomorphic $G$-space such that the analytic Hilbert quotient $\pi: X \rightarrow X\hq
G$ exists and such that $X\hq G$ is a projective algebraic variety. Assume that there are two
algebraic $G$-varieties $X_1$ and $X_2$ with algebraic Hilbert quotients $\pi_j: X_j \rightarrow X_j
\hq G$, $j =1,2$ such that $X$ is $G$-equivariantly biholomorphic to both $X_1^h$ and $X_2^h$. Then, $X_1$ and $X_2$ are isomorphic as algebraic $G$-varieties.
\end{prop}
\vspace*{-3mm}
\begin{proof}
 Let $\phi : (X_1)^h \rightarrow (X_2)^h$ be the holomorphic map given by the two isomorphisms $\phi_j: X \to X_j$ for $j =1,2$ and let $\Gamma \subset X_1 \times X_2$ be its graph. The set $\Gamma$ is $G$-invariant, analytic, and $G$-equivariantly biholomorphic to $X$ via $\Phi: X \to \Gamma$, $\Phi = \phi_1 \times \phi_2$. We will show that $\Gamma$ is an algebraic subvariety of $X_1
\times X_2$.

The $G$-invariant
algebraic map $\pi_1 \times \pi_2 : X_1 \times X_2 \rightarrow X_1\hq G \times X_2\hq G$ is affine. Applying \cite{BBExistenceOfQuotients}, Corollary 5, it follows that there exists an algebraic Hilbert quotient $\Pi: X_1 \times X_2 \rightarrow (X_1 \times X_2)\hq G \rdefiniere Q$ for the $G$-action on $X_1 \times X_2$. Consider the following commutative diagram
\[\begin{xymatrix}{
X \ar[r]^<<<<<<{\Phi} \ar[d]_{\pi}&  X_1 \times X_2 \ar[d]^{\Pi}\\
X\hq G \ar[r]^{\widebar\Phi} &   Q.
}
  \end{xymatrix}
\]
Since $X\hq G$ is projective, it follows that $\Pi(\Gamma)= \widebar \Phi (X\hq G)$ is a
compact algebraic subvariety of $Q$ isomorphic to $X\hq G$. Its preimage $Y \definiere \Pi^{-1} (\Pi
(\Gamma))$ is an algebraic $G$-subvariety of $X_1 \times X_2$. The map
$\pi_Y\definiere \Pi|_Y: Y \rightarrow \Pi(Y)=Y\hq G$ is an algebraic Hilbert quotient, $Y\hq G
\cong X\hq G$ is projective, and $\Gamma \subset Y$ is a $G$-invariant analytic subset such that
$\pi_Y(\Gamma) = Y\hq G$. The Algebraicity Theorem for invariant analytic sets, Theorem
\ref{invariantalgebraic}, implies that $\Gamma \subset Y \subset X_1 \times X_2$ is algebraic. We conclude that $\phi$ is a regular map.
\end{proof}
The proof of the Algebraicity Theorem, Theorem~\ref{2}, now follows by applying Propositions \ref{existsalgstructure} and \ref{twoquotients}.

\begin{rem}
The claim of Theorem \ref{2} is no longer true, if the quotient is not
assumed to be projective, as illustrated by the following example.
\end{rem}
\begin{ex}
Let $A \subset \C^2$ be any non-algebraic, analytic subset of $\C^2$. If $\varphi: A \to \C^2$ is the inclusion of $A$ into $\C^2$, we consider the space $X \definiere \C^2 \cup_\varphi  \C^2 = (\C^2 \sqcup \C^2)/\negthickspace\sim$, where $\sqcup$ denotes disjoint union, and
where $x \sim y$, if and only if $y = \varphi(x)$. Let $Y \definiere \C^2 \sqcup \C^2$
and let $p: Y \to X$ be the quotient map. For an open subset $U$ in $X$ we define
\[\mathscr{H}_{X}(U) \definiere \{f \in \mathscr{H}_Y(p^{-1}(U)) \mid f
\text{ constant on $\sim$-equivalence classes}\}.\]
By a Theorem of H.\ Cartan \cite{CartanQuotients} the ringed space $(X,
\mathscr{H}_{X})$ is a complex space, and the projection map $p: Y \to X$ is holomorphic. Consider the holomorphic map $\sigma: Y \to Y$ which sends a point $x$ in one of the connected components $\C^2$ of $Y$ to the corresponding point in
the other connected component of $Y = \C^2 \sqcup \C^2$. The automorphism $\sigma$
generates a holomorphic action of $\Z_2 = \Z/2\Z$ on $Y$. This action descends to a holomorphic $\Z_2$-action on $X$ making $p: Y \to X$ equivariant. Consider the holomorphic map $\psi: Y \to \C^2$ which maps a point $x \in Y$ to the corresponding
point in $\C^2$. The map $\psi$ is constant on equivalence classes and hence descends to a
holomorphic map $\pi: X \to \C^2$. This is an analytic Hilbert quotient for the $\Z_2$-action on
$X$.

We claim that there exists no algebraic structure on $X$ making $X$ into an algebraic
$\Z_2$-variety in such a way that the map $\pi: X \to \C^2$ becomes the algebraic Hilbert quotient for
the action of $\Z_2$ on $X$. Indeed, suppose that this was possible. Then, $X$ would be an affine
variety and the map $\pi: X \to \C^2$ would be finite. Therefore, the branch locus $R_\pi
\subset \C^2$ of $\pi$ would be an algebraic subvariety of $\C^2$. However,
set-theoretically, we have $R_\pi = A$, a contradiction.
\end{ex}

\section{Analytic Hilbert quotients as algebraic Hilbert quotients}\label{openproblems}
We conclude our study by discussing connections between the results on momentum map quotients  obtained in the first part of this note (Theorems~\ref{1} and \ref{unstablepointsarealgebraic}) and the Algebraicity Theorem, Theorem~\ref{2}.

We recall the setup of Theorems~\ref{1} and \ref{unstablepointsarealgebraic}. Let $G$ be the complexification of the compact Lie group $K$. Let $X$ be a $G$-irreducible algebraic Hamiltonian $G$-variety with momentum map $\mu: X \to
\mathfrak{k}^*$. Assume that $X$ has only $1$-rational singularities and that $\mu^{-1}(0)$ is
compact. Then, Theorem \ref{1} and Theorem \ref{unstablepointsarealgebraic} imply that
$X(\mu)\hq G$ is projective algebraic and that $X(\mu)$ is an algebraically Zariski-open subset in
$X$, hence itself an algebraic $G$-variety. It now follows from
Theorem \ref{2} that there exists a
quasi-projective algebraic $G$-variety $Z$ and a $G$-equivariant biholomorphic map $\psi: X(\mu)^h
\to Z^h$. It is a natural question whether $\psi$ preserves the algebraic structure on $X(\mu)^h$.

The map $\psi$ is induced by an isomorphism of algebraic varieties if and only if $\pi: X(\mu) \to X(\mu)\hq G$ is an algebraic Hilbert quotient. The following result states that this is true under an additional regularity assumption.

\begin{prop}\label{momentumaffine}
Let $K$ be a compact Lie group and let $G=K^\C$ be its
complexification. Let $X$ be a $G$-irreducible algebraic
Hamiltonian $G$-variety with momentum map $\mu: X \to \mathfrak{k}^*$. Assume that $X$ has only
$\mathit{1}$-rational singularities and that $\mu^{-1}(0)$ is compact. Assume in addition that all stabiliser
groups of points in $X(\mu)$ are finite. Then, $\pi: X(\mu) \to X(\mu)\hq G$ is an algebraic Hilbert quotient. In particular, the map $\psi$ is biregular.
\end{prop}
\vspace*{-3mm}
\begin{proof}
Since the stabiliser groups of elements in $X(\mu)$ are finite, all $G$-orbits in
$X(\mu)$ are closed in $X(\mu)$, and the quotient $\pi: X(\mu) \to X(\mu)\hq G$ is geometric. It follows (cf.\ the proof of Proposition~\ref{prop1}) that the quotient map $\pi: X(\mu) \to X(\mu)\hq G$ is rational. Since $\pi$ is a priori holomorphic, we conclude that it is regular. The existence of local slices for the
action of $G$ on $X(\mu)$ together with the existence of a geometric quotient for the $G$-action on $X(\mu)$ implies that $G$ acts properly on $X(\mu)$. It then follows from \cite{KollarQuotients} that $\pi$ is an algebraic Hilbert quotient.
\end{proof}
\begin{rems}
a) If $X$ is smooth, the finiteness of the isotropy groups required in Proposition \ref{momentumaffine} is equivalent to $0$ being a regular value of the momentum map $\mu$.

b) If, in addition to the assumptions made in Proposition~\ref{momentumaffine}, $X$ is $\Q$-factorial, there exists a $G$-linearised line bundle $L$ on $X$ such that $X(\mu)$ coincides with the set $X(L)$ of semistable points with respect to $L$, cf.\ \cite{MumfordGIT}, Chap.\ 1, \S 4, Conv 1.13. Hence, in this setup the analytic theory of momentum map quotients coincides with Mumford's algebraic Geometric Invariant Theory.

c) If $X$ is a smooth projective algebraic Hamiltonian $G$-variety, the assertion of Proposition~\ref{momentumaffine} is true without the assumption on stabiliser groups, see \cite{MomentumProjectivity}. It is however an open question if the quotient map $\pi: X(\mu) \to X(\mu)\hq G$ of a general algebraic Hamiltonian $G$-variety is an algebraic Hilbert quotient.
\end{rems}
\def\cprime{$'$} \def\polhk#1{\setbox0=\hbox{#1}{\ooalign{\hidewidth
  \lower1.5ex\hbox{`}\hidewidth\crcr\unhbox0}}}
  \def\polhk#1{\setbox0=\hbox{#1}{\ooalign{\hidewidth
  \lower1.5ex\hbox{`}\hidewidth\crcr\unhbox0}}}
\providecommand{\bysame}{\leavevmode\hbox to3em{\hrulefill}\thinspace}
\providecommand{\MR}{\relax\ifhmode\unskip\space\fi MR }
\providecommand{\MRhref}[2]{%
  \href{http://www.ams.org/mathscinet-getitem?mr=#1}{#2}
}
\providecommand{\href}[2]{#2}

\vfill
Daniel Greb\\
Albert-Ludwigs-Universit\"at\\
Mathematisches Institut\\
Abteilung f\"ur Reine Mathematik\\
79104 Freiburg im Breisgau\\
Germany

\verb"daniel.greb@math.uni-freiburg.de"
\end{document}